\documentclass[11pt,a4paper]{article}
  \def\SvZfontcode{8}
  \def\SvZslantedGreekCapitals{1}
  \usepackage{amsmath,amsthm,amscd}% AMS macros and such

\def\SvZrequireslantedRedef{0}
\ifcase\SvZfontcode
\usepackage{bm}% bold symbols
\usepackage{amssymb}
\def\SvZrequireslantedRedef{1}
\or
\usepackage{fouriernc}
\usepackage{bm}% bold symbols
\def\SvZrequireslantedRedef{1}
\or
\usepackage{fourier}
\usepackage{bm}% bold symbols
\def\SvZrequireslantedRedef{1}
\or
\usepackage{amssymb}
\if\SvZslantedGreekCapitals 1
\usepackage[slantedGreek]{mathptmx}
\else
\usepackage{mathptmx}
\fi
\DeclareMathAlphabet{\bm}{OT1}{ptm}{b}{it} % Apparently this creates boldface. It is better than bm, but doesn't make symbols bold.
\or
\usepackage{kmath,kerkis}
\usepackage{bm}% bold symbols
\def\SvZrequireslantedRedef{1}
\or
\if\SvZslantedGreekCapitals 1
\usepackage[sc,slantedGreek]{mathpazo}
\else
\usepackage[sc]{mathpazo}
\fi
\usepackage{bm}% bold symbols
\or
\usepackage[adobe-utopia]{mathdesign}
\usepackage{bm}% bold symbols
\or
\usepackage[urw-garamond]{mathdesign}
\usepackage{bm}% bold symbols
\or
\usepackage[bitstream-charter]{mathdesign}
\usepackage{bm}% bold symbols
\or
\usepackage{lmodern}
\usepackage{bm}% bold symbols
\def\SvZrequireslantedRedef{1}
\or
\usepackage{arev}
\fi

\if\SvZrequireslantedRedef 1
\if\SvZslantedGreekCapitals 1
\renewcommand{\Gamma}{\varGamma}
\renewcommand{\Delta}{\varDelta}
\renewcommand{\Theta}{\varTheta}
\renewcommand{\Lambda}{\varLambda}
\renewcommand{\Xi}{\varXi}
\renewcommand{\Pi}{\varPi}
\renewcommand{\Sigma}{\varSigma}
\renewcommand{\Upsilon}{\varUpsilon}
\renewcommand{\Phi}{\varPhi}
\renewcommand{\Psi}{\varPsi}
\renewcommand{\Omega}{\varOmega}
\fi
\fi
\renewcommand{\phi}{\varphi}
\reversemarginpar

\usepackage{url}
\usepackage{listings}
\usepackage{kbordermatrix}
\usepackage{booktabs}
\usepackage{graphicx}
\usepackage{color}

\ifx\SvZinPresentation\undefined
\usepackage{float}
\usepackage{longtable}
\usepackage[margin=10pt,labelfont=bf,textfont=it]{caption}

\usepackage[colorlinks]{hyperref}
\renewcommand{\theenumi}{\textit{\roman{enumi}}}

\newcommand{\mathds}{\mathbb}

\DeclareMathOperator{\rank}{rank}

\newcommand{\Q}{ \mathds{Q} }
\newcommand{\R}{ \mathds{R} }

\newcommand{\Z}{ \mathds{Z} }
\newcommand{\C}{ \mathds{C} }
\newcommand{\N}{\mathds{N}}
\newcommand{\field}{\mathds{F}}
\newcommand{\group}{G}
\newcommand{\bracketring}{\mathds{B}}
\DeclareMathOperator{\GF}{GF}
\newcommand{\ring}{R}
\newcommand{\parf}{\ensuremath{\mathbb{P}}} % partial field
\newcommand{\reg}{\uniform_0}
\newcommand{\dyadic}{\mathds{D}}
\newcommand{\nreg}{\uniform_1}
\newcommand{\psru}{\mathds{S}}
\newcommand{\sru}{\ensuremath{\sqrt[6]{1}}}
\newcommand{\golmean}{\mathds{G}}%ensuremath{\pf_{\mathrm{gm}}}}
\newcommand{\golrat}{\mathds{G}}%ensuremath{\pf_{\mathrm{gm}}}}
\newcommand{\splittable}{\mathds{Y}}
\newcommand{\dowfour}{\mathds{W}}
\newcommand{\ger}{\mathds{GE}}
\newcommand{\gauss}{\hydra_2}
\newcommand{\hydra}{\mathds{H}}
\newcommand{\cyclo}{\mathds{K}}

\newcommand{\uniform}{\mathds{U}}

\DeclareMathOperator{\crat}{Cr}% cross ratios
\newcommand{\Crat}{\crat}
\DeclareMathOperator{\assoc}{Asc}% associates (symmetry group of fund. elt. under 1-p and 1/p)
\DeclareMathOperator{\fun}{{\cal F}}
\newcommand{\lift}{\mathds{L}}

\renewcommand{\tilde}{\widetilde}
\renewcommand{\hat}{\widehat}
\renewcommand{\bar}{\overline}

\newcommand{\ignore}[1]{}

\DeclareMathOperator{\sign}{sgn}

\DeclareMathOperator{\PG}{PG} % projective geometry
\DeclareMathOperator{\AG}{AG} % projective geometry
\newcommand{\deltadot}{Q^+}%{\mathaccent\cdot\triangle}

\let\Oldsetminus\setminus
\newcommand{\delete}{\ensuremath{\!\Oldsetminus\!}}
\newcommand{\contract}{\ensuremath{\!/}}

\newcommand{\minorof}{\ensuremath{\preceq}}

\newcommand{\bip}{G}

\ifx\SvZinPresentation\undefined
\newtheorem{theorem}{Theorem}[section]

\newtheorem{lemma}[theorem]{Lemma}
\newtheorem{proposition}[theorem]{Proposition}
\newtheorem{definition}[theorem]{Definition}
\newtheorem{corollary}[theorem]{Corollary}

\newtheorem{claim}{Claim}[theorem] % For partial results within a proof
\newenvironment{claimenv}{\list{}{\rightmargin0pt\leftmargin10pt\topsep0pt}\item[]}{\endlist}

\newenvironment{subproof}{\begin{claimenv}\begin{proof}}{\end{proof}\end{claimenv}}
\fi
\newtheorem{conjecture}[theorem]{Conjecture}
\newtheorem{question}[theorem]{Question}
 
\newcommand{\Dutchvon}[2]{#2}

\usepackage[longnamesfirst,numbers]{natbib}
\shortcites{oriented}

\begin{document}
  %%%%%%%%%%%%%%%%%%%%%%%%%%%%%%%%%%%%%%%%%%%%%%%%%%%%%%%%%%%%%%%%%%%%%
  \title{Confinement of matroid representations to subsets of partial fields}
  %%%%%%%%%%%%%%%%%%%%%%%%%%%%%%%%%%%%%%%%%%%%%%%%%%%%%%%%%%%%%%%%%%%%%
  \author{R. A. Pendavingh and S. H. M. van Zwam\thanks{E-mail: \url{rudi@win.tue.nl}, \url{Stefan.van.Zwam@cwi.nl}. This research was supported by NWO, grant 613.000.561. Parts of this paper have appeared in the second author's PhD thesis \cite{vZ09}.}}
  \maketitle
  \abstract{
    Let $M$ be a matroid representable over a (partial) field $\parf$ and $B$ a matrix representable over a sub-partial field $\parf'\subseteq\parf$. We say that $B$ \emph{confines} $M$ to $\parf'$ if, whenever a $\parf$-representation matrix $A$ of $M$ has a submatrix $B$, $A$ is a scaled $\parf'$-matrix. We show that, under some conditions on the partial fields, on $M$, and on $B$, verifying whether $B$ confines $M$ to $\parf'$ amounts to a finite check. A corollary of this result is Whittle's Stabilizer Theorem~\cite{Whi96b}.

    A combination of the Confinement Theorem and the Lift Theorem from \citet{PZ08lift} leads to a short proof of Whittle's characterization of the matroids representable over $\GF(3)$ and other fields~\cite{Whi97}.

    We also use a combination of the Confinement Theorem and the Lift Theorem to prove a characterization, in terms of representability over partial fields, of the 3-connected matroids that have $k$ inequivalent representations over $\GF(5)$, for $k = 1, \ldots, 6$.

    Additionally we give, for a fixed matroid $M$, an algebraic construction of a partial field $\parf_M$ and a representation matrix $A$ over $\parf_M$ such that every representation of $M$ over a partial field $\parf$ is equal to $\phi(A)$ for some homomorphism $\phi:\parf_M\rightarrow\parf$. Using the Confinement Theorem we prove an algebraic analog of the theory of free expansions by \citet{GOVW02}.
  }

%%%%%%%%%%%%%%%%%%%%%%%%%%%%%%%%%%%%%%%%%%%%%%%%%%%%%%%%%%%%%%%%%%%%%
\section{Introduction}
%%%%%%%%%%%%%%%%%%%%%%%%%%%%%%%%%%%%%%%%%%%%%%%%%%%%%%%%%%%%%%%%%%%%%
Questions regarding the representability of matroids pervade matroid theory. A famous theorem is the characterization of regular matroids due to Tutte. We say that a matrix over the real numbers is \emph{totally unimodular} if the determinant of every square submatrix is in the set $\{-1,0,1\}$.
\begin{theorem}[\citet{Tut65}]\label{thm:reg}
Let $M$ be a matroid. The following are equivalent:
\begin{enumerate}
  \item $M$ is representable over $\GF(2)$ and some field that does not have characteristic 2;
  \item  $M$ is representable over $\R$ by a totally unimodular matrix;
  \item $M$ is representable over every field.
\end{enumerate}
\end{theorem}
Whittle gave a similar characterization of the matroids representable over $\GF(3)$ and some other field. We say that a matrix over the real numbers is \emph{dyadic} if the determinant of every square submatrix is in the set $\{0\}\cup\{\pm 2^k \mid k \in \Z\}$. We say that a matrix over the complex numbers is \emph{sixth-roots-of-unity} ($\sqrt[6]{1}$) if the determinant of every square submatrix is in the set $\{0\}\cup\{ \zeta^l \mid l \in \Z\}$, where $\zeta$ is a root of $x^2-x+1=0$ (so $\zeta^6 = 1$).
\begin{theorem}[\citet{Whi97}]\label{thm:classsintro}
  Let $M$ be a 3-connected matroid that is representable over $\GF(3)$ and some field that is not of characteristic 3. Then at least one of the following holds:
  \begin{enumerate}
    \item\label{it:whi1} $M$ is representable over $\R$ by a dyadic matrix;
    \item\label{it:whi2} $M$ is representable over $\C$ by a $\sru$-matrix.
  \end{enumerate}
\end{theorem}
Whittle's characterization was, in fact, more precise. He also characterized the matroids as in \eqref{it:whi1},\eqref{it:whi2} by the set of fields over which $M$ is representable. In \cite{PZ08lift} we proved the Lift Theorem, a general theorem from which Whittle's results of the latter type follow. But the Lift Theorem is not sufficient to prove that Whittle's classification is complete. In this paper we will fill this gap by proving the Confinement Theorem. Using this we will be able to give a comparatively short proof of Whittle's theorem.

The Confinement Theorem has other applications. For instance, Whittle's Stabilizer Theorem~\cite{Whi96b} is a corollary of it. Semple and Whittle's~\cite{SW96b} result that every representable matroid with no $U_{2,5}$- and no $U_{3,5}$-minor is either binary or ternary can be proven with it, again by combining it with the Lift Theorem. We were led to the Confinement Theorem by our study of matroids with inequivalent representations over $\GF(5)$. \citet{OVW95} proved that a 3-connected quinary\footnote{Some authors prefer the word \emph{quinternary}} matroid never has more than 6 inequivalent representations. Using the Lift Theorem and the Confinement Theorem we were able to extend that result as follows:
\begin{theorem}\label{thm:quinary}
  Let $M$ be a 3-connected quinary matroid. Then $M$ has at most 6 inequivalent representations over $\GF(5)$. Moreover, the following hold:
  \begin{enumerate}
    \item\label{quin:two}If $M$ has at least two inequivalent representations over $\GF(5)$, then $M$ is representable over $\C$, over $\GF(p^2)$ for all primes $p \geq 3$, and over $\GF(p)$ when $p \equiv 1 \mod 4$.
    \item\label{quin:three}If $M$ has at least three inequivalent representations over $\GF(5)$, then $M$ is representable over every field with at least five elements.
    \item\label{quin:four}If $M$ has at least four inequivalent representations over $\GF(5)$, then $M$ is not binary and not ternary.
    \item\label{quin:five}If $M$ has at least five inequivalent representations over $\GF(5)$, then $M$ has six inequivalent representations over $\GF(5)$.
  \end{enumerate}
\end{theorem}
We note here that we proved \eqref{quin:two} in \cite{PZ08lift}, and that \eqref{quin:four} is a special case of a result by \citet{Whi96}.

We will now give a more detailed overview of the contents of this paper. The framework for our results is the theory of partial fields, introduced by \citet{SW96}. A partial field is an algebraic structure resembling a field, but in which addition is not always defined. Semple and Whittle developed a theory of matroids representable over partial fields. In \cite{PZ08lift} we gave a proof of the theorem by Vertigan that partial fields can be obtained as the restriction of a ring to a subgroup of its group of units. In this paper we will use this as definition of a partial field, rather than the axiomatic setup by Semple and Whittle. We repeat, and sometimes extend, the relevant definitions and results from \citet{SW96} and \citet{PZ08lift} in Section~\ref{sec:pfintro} of this paper. We note here that \citet{Cam68} (translated and updated in \cite{Cam06}) introduced a class of matrices equivalent to our partial-field matrices. His results have almost no overlap with ours.

Sometimes a matroid that is representable over a partial field $\parf$ is in fact also representable over a sub-partial field $\parf'\subseteq \parf$. Let $M,N$ be matroids such that $N$ is a minor of $M$. Suppose that, whenever a $\parf$-representation $A$ of $M$ contains a scaled $\parf'$-representation of $N$, $A$ itself is a scaled $\parf'$-representation of $M$. Then we say that $N$ \emph{confines} $M$ to $\parf'$. The following theorem reduces verifying if $N$ confines $M$ to a finite check.
\begin{theorem}\label{thm:confineintro}
  Let $\parf, \parf'$ be partial fields such that $\parf'$ is an induced sub-partial field of $\parf$. Let $M, N$ be 3-connected matroids such that $N$ is a minor of $M$. Then exactly one of the following holds:
  \begin{enumerate}
    \item $N$ confines $M$ to $\parf'$;
    \item $M$ has a 3-connected minor $M'$ such that
    \begin{itemize}
      \item $N$ does not confine $M'$ to $\parf'$;
      \item $N$ is isomorphic to $M'\contract x$, $M'\delete y$, or $M'\contract x \delete y$ for some $x,y\in E(M')$;
      \item If $N$ is isomorphic to $M'\contract x \delete y$ then at least one of $M'\contract x, M'\delete y$ is 3-connected.
    \end{itemize}
  \end{enumerate}
\end{theorem}
We will define \emph{induced} sub-partial fields in Subsection~\ref{ssec:subpf}, but note here that if a sub-partial field is induced then $p+q \in \parf'$ whenever $p,q \in \parf'$ and $p+q \in\parf$.
The main result of this paper, the Confinement Theorem (Theorem~\ref{thm:confinement}) is stated in terms of individual representation matrices. Theorem~\ref{thm:confineintro} is a direct corollary.

The Confinement Theorem closely resembles several results related to inequivalent representations of matroids. These results are Whittle's Stabilizer Theorem~\cite{Whi96b}, the extension to universal stabilizers by \citet{GOVW98}, and the theory of free expansions by the same authors \cite{GOVW02}. In fact, Whittle's Stabilizer Theorem is a corollary of the Confinement Theorem. To prove this we use the observation that multiple representations of a matroid can be combined into a single representation over a bigger partial field.

In most of our applications we combine the Confinement Theorem with the Lift Theorem from \cite{PZ08lift}. We first compute the lift partial field for a class of $\parf$-representable matroids. Then we use the Confinement Theorem to split off certain induced sub-partial fields from this lift partial field. This approach can be used, for instance, to give an alternative proof of Whittle's~\cite{Whi95,Whi97} characterization of the matroids representable over $\GF(3)$ and other fields. This proof can be found in Subsection~\ref{ssec:ternary}.

In Section~\ref{sec:pfdesign} we shift our focus to more algebraic techniques. A question that matroid theorists have considered is, for a fixed matroid $M$, the determination of all primes $p$ such that $M$ is representable over some field of characteristic $p$. \citet{Vam71}, \citet{Wte75a,Wte75b}, and \citet{Fen84} all answer this question by constructing, for a fixed matroid $M$, a ring $\ring_M$, such that representations of $M$ over a field $\field$ are related to ring homomorphisms $\ring_M\rightarrow\field$. Recently \citet{BV03} gave an algorithm to compute the set of characteristics for a given matroid by computing certain Gr\"obner bases over the integers. We refer to Oxley \cite[{Section 6.8}]{oxley} and White \cite[{Chapter 1}]{Wte87} for more details on this subject.

In this paper we strengthen the construction by \citet{Wte75a} to give a partial field $\parf$ and a matrix $A$ with entries in $\parf$, such that every representation of $M$ over a partial field $\parf'$ is equivalent to $\phi(A)$ for some partial-field homomorphism $\phi:\parf\rightarrow\parf'$.
The advantage of our approach over that of the papers mentioned above is that the matrix $A$ is itself a representation of $M$ over $\parf$, rather than an object from which representations can be created. \citet{Fen84} created a smaller ring that retained the universality of White's construction. Likewise we will show that a sub-partial field $\parf_M\subseteq \parf$ suffices to represent $M$. We will prove that $\parf_M$ is the smallest such partial field. We call $\parf_M$ the \emph{universal partial field of $M$}.

In Subsection~\ref{ssec:universal} we compute the universal partial field for two classes of matroids, and show that the partial fields studied in \citet{PZ08lift} are all universal. We conclude Section~\ref{sec:pfdesign} with another corollary of the Confinement Theorem, which we call the Settlement Theorem.

In Subsection~\ref{sec:GFfive} we use the combined power of the Lift Theorem from~\cite{PZ08lift}, the Confinement Theorem, and the algebraic constructions to prove Theorem~\ref{thm:quinary}. First we use the theory of universal partial fields to characterize the number of representations of quinary matroids with no $U_{2,5}$- and $U_{3,5}$-minor. Then we construct, for each $k \in \{ 1,\ldots,6\}$, a partial field $\hydra_k$ over which a 3-connected quinary matroid $M$ with a $U_{2,5}$- or $U_{3,5}$-minor is representable if and only if it has at least $k$ inequivalent representations over $\GF(5)$. The result then follows by considering the homomorphisms $\hydra_k\rightarrow\field$ for fields $\field$.

We conclude in Section~\ref{sec:questions} with a number of unsolved problems. In an appendix we list all partial fields discussed in this paper and in \cite{PZ08lift}, along with some of their properties.

%%%%%%%%%%%%%%%%%%%%%%%%%%%%%%%%%%%%%%%%%%%%%%%%%%%%%%%%%%%%%%%%%%%%%
\section{Preliminaries}\label{sec:pfintro}
%%%%%%%%%%%%%%%%%%%%%%%%%%%%%%%%%%%%%%%%%%%%%%%%%%%%%%%%%%%%%%%%%%%%%
In Subsections~\ref{ssec:not}--\ref{ssec:crat} we define partial fields and summarize the relevant definitions and results from \citet{SW96} and \citet{PZ08lift}. After that we give some extra definitions and some first new results.

%%%%%%%%%%%%%%%%%%%%%%%%%%%%%%%%%%%%%%%%%%%%%%%%%%%%%%%%%%%%%%%%%%%%%
\subsection{Notation}\label{ssec:not}
%%%%%%%%%%%%%%%%%%%%%%%%%%%%%%%%%%%%%%%%%%%%%%%%%%%%%%%%%%%%%%%%%%%%%
If $S, T$ are sets, and $f:S\rightarrow T$ is a function, then we define
\begin{align}
  f(S):= \{f(s) \mid s \in S\}.
\end{align}
We denote the restriction of $f$ to $S'\subseteq S$ by $f|_{S'}$. We may simply write $e$ instead of the singleton set $\{e\}$.

If $S$ is a subset of nonzero elements of some group, then $\langle S \rangle$ is the subgroup generated by $S$. If $S$ is a subset of elements of a ring, then $\langle S \rangle$ denotes the \emph{multiplicative} subgroup generated by $S$. All rings are commutative with identity. The group of elements with a multiplicative inverse (the \emph{units}) of a ring $\ring$ is denoted by $\ring^*$. If $\ring$ is a ring and $S$ a set of symbols, then we denote the polynomial ring over $\ring$ on $S$ by $\ring[S]$.

Our graph-theoretic notation is mostly standard. All graphs encountered are simple. We use the term \emph{cycle} for a simple, closed path in a graph, reserving \emph{circuit} for a minimal dependent set in a matroid. An undirected edge (directed edge) between vertices $u$ and $v$ is denoted by $uv$ and treated as a set $\{u, v\}$ (an ordered pair $(u,v)$). We define $\delta(v) := \{e \in E(G) \mid e = uv \textrm{ for some } u \in V\}$. If $G = (V,E)$ and $V'\subseteq V$, then we denote the induced subgraph on $V'$ by $G[V']$. For $S, T \subseteq V$ we denote by $d_G(S,T)$ the length of a shortest $S-T$ path in $G$.

For matroid-theoretic concepts we follow the notation of \citet{oxley}. Familiarity with the definitions and results in that work is assumed.

%%%%%%%%%%%%%%%%%%%%%%%%%%%%%%%%%%%%%%%%%%%%%%%%%%%%%%%%%%%%%%%%%%%%%
\subsection{Partial fields}
%%%%%%%%%%%%%%%%%%%%%%%%%%%%%%%%%%%%%%%%%%%%%%%%%%%%%%%%%%%%%%%%%%%%%

A \emph{partial field} is a pair $\parf = (\ring, \group)$, where $\ring$ is a commutative ring with identity and $\group$ is a subgroup of the group of units $\ring^*$ of $\ring$ such that $-1 \in G$. If $1 = 0$ in $\ring$ then we say the partial field is $\emph{trivial}$. When $\parf$ is referred to as a set, then it is the set $\group \cup \{0\}$.
We define $\parf^* := \group$. Every field $\field$ can be considered as a partial field $(\field, \field^*)$.

A useful construction is the following.
\begin{definition}If $\parf_1 = (\ring_1, \group_1), \parf_2 = (\ring_2, \group_2)$ are partial fields, then the \emph{direct product} is
\begin{align}
  \parf_1\otimes\parf_2 := (\ring_1 \times \ring_2, \group_1\times \group_2).
\end{align}
\end{definition}
Recall that in the product ring addition and multiplication are defined componentwise. It is readily checked that $\parf_1\otimes\parf_2$ is again a partial field.

A function $\phi: \parf_1 \rightarrow \parf_2$ is a \emph{partial field homomorphism} if
\begin{enumerate}
  \item $\phi(1) = 1$;
  \item for all $p, q \in \parf_1$, $\phi(pq) = \phi(p)\phi(q)$;
  \item for all $p, q \in \parf_1$ such that $p+q \in \parf_1$, $\phi(p) + \phi(q) = \phi(p+q)$.
\end{enumerate}
If $\parf_1 = (\ring_1, \group_1)$, $\parf_2 = (\ring_2, \group_2)$, and $\phi:\ring_1\rightarrow\ring_2$ is a ring homomorphism such that $\phi(\group_1)\subseteq\group_2$, then the restriction of $\phi$ to $\parf_1$ is obviously a partial field homomorphism. However, not every partial field homomorphism extends to a homomorphism between the rings. We refer to \cite[Theorem 5.3]{PZ08lift} for the precise relation between partial field homomorphisms and ring homomorphisms.

Suppose $\parf, \parf_1, \parf_2$ are partial fields such that there exist homomorphisms $\phi_1:\parf\rightarrow \parf_1$ and $\phi_2:\parf\rightarrow\parf_2$. Then we define $\phi_1\otimes \phi_2: \parf \rightarrow \parf_1\otimes\parf_2$ by $(\phi_1\otimes\phi_2)(p) := (\phi_1(p),\phi_2(p))$.
\begin{lemma}[{\cite[{Lemma~2.18}]{PZ08lift}}]$\phi_1\otimes\phi_2$ is a partial field homomorphism.
\end{lemma}

A partial field isomorphism $\phi:\parf_1\rightarrow\parf_2$ is a bijective homomorphism with the additional property that $\phi(p+q)\in\parf_2$ if and only if $p+q\in\parf_1$. If $\parf_1$ and $\parf_2$ are isomorphic then we denote this by $\parf_1 \cong \parf_2$. A partial field automorphism is an isomorphism $\phi:\parf\rightarrow\parf$.

%%%%%%%%%%%%%%%%%%%%%%%%%%%%%%%%%%%%%%%%%%%%%%%%%%%%%%%%%%%%%%%%%%%%%
\subsection{Partial-field matrices}
%%%%%%%%%%%%%%%%%%%%%%%%%%%%%%%%%%%%%%%%%%%%%%%%%%%%%%%%%%%%%%%%%%%%%
Recall that formally, for ordered sets $X$ and $Y$, an $X\times Y$ matrix $A$ with entries in a partial field $\parf$ is a function $A:X\times Y \rightarrow \parf$. If $X'\subseteq X$ and $Y'\subseteq Y$, then we denote by $A[X',Y']$ the submatrix of $A$ obtained by deleting all rows and columns in $X\setminus X'$, $Y\setminus Y'$. If $Z$ is a subset of $X\cup Y$ then we define $A[Z] := A[X\cap Z, Y \cap Z]$. Also, $A-Z := A[X\setminus Z, Y\setminus Z]$. If $X = \{1, \ldots, r\}$ then we say that $A$ is an $r\times Y$ matrix.

An $X\times Y$ matrix $A$ with entries in $\parf$ is a \emph{$\parf$-matrix} if $\det(A')\in\parf$ for every square submatrix $A'$ of $A$.

\begin{definition}\label{def:pivot}Let $A$ be an $X\times Y$ $\parf$-matrix, and let $x \in X, y \in Y$ be such that $A_{xy} \neq 0$. Then we define $A^{xy}$ to be the $((X\setminus x) \cup y) \times ((Y\setminus y) \cup x)$ matrix given by
\begin{align}
  (A^{xy})_{uv} = \left\{ \begin{array}{ll}
    A_{xy}^{-1} \quad & \textrm{if } uv = yx\\
    A_{xy}^{-1} A_{xv} & \textrm{if } u = y, v\neq x\\
    -A_{xy}^{-1} A_{uy} & \textrm{if } v = x, u \neq y\\
    A_{uv} - A_{xy}^{-1} A_{uy} A_{xv} & \textrm{otherwise.}
  \end{array}\right.
\end{align}
\end{definition}
We say that $A^{xy}$ is obtained from $A$ by \emph{pivoting} over $xy$. The pivot operation can be interpreted as adding an $X\times X$ identity matrix to $A$, doing row reduction, followed by a column exchange, and finally removing the new identity matrix.

\begin{lemma}[{\cite[Lemma 2.6]{PZ08lift}}]\label{lem:detpivot}
  Let $A$ be an $X\times Y$ matrix with entries in $\parf$ such that $|X|=|Y|$ and $A_{xy} \neq 0$. If $\det(A^{xy}-\{x,y\})\in\parf$ then $\det(A)\in\parf$, and
  \begin{align}
  \det(A)= (-1)^{x+y}A_{xy} \det(A^{xy}-\{x,y\}).
  \end{align}
\end{lemma}
\begin{definition}\label{def:matops}Let $A$ be an $X\times Y$ $\parf$-matrix. We say that $A'$ is a \emph{minor} of $A$ (notation: $A' \minorof A$) if $A'$ can be obtained from $A$ by a sequence of the following operations:
\begin{enumerate}
  \item Multiplying the entries of a row or column by an element of $\parf^*$;
  \item Deleting rows or columns;
  \item Permuting rows or columns (and permuting labels accordingly);
  \item Pivoting over a nonzero entry.
\end{enumerate}
\end{definition}
Be aware that in linear algebra a ``minor of a matrix'' is defined differently. We use Definition~\ref{def:matops} because of its relation with matroid minors, which will be explained in the next section. For a determinant of a square submatrix we use the word \emph{subdeterminant}.
\begin{proposition}[{\cite[Proposition 3.3]{SW96}}]\label{prop:pmatops}Let $A$ be a $\parf$-matrix. Then $A^T$ is also a $\parf$-matrix. If $A'\minorof A$ then $A'$ is a $\parf$-matrix.
\end{proposition}

Let $A$ be an $X\times Y$ $\parf$-matrix, and let $A'$ be an $X'\times Y'$ $\parf$-matrix. Then $A$ and $A'$ are \emph{isomorphic} if there exist bijections $f:X\rightarrow X'$, $g:Y\rightarrow Y'$ such that for all $x\in X, y \in Y$, $A_{xy} = A'_{f(x)g(y)}$.

Let $A$, $A'$ be $X\times Y$ $\parf$-matrices. If $A'$ can be obtained from $A$ by scaling rows and columns by elements from $\parf^*$, then we say that $A$ and $A'$ are \emph{scaling-equivalent}, which we denote by $A\sim A'$.

Let $A$ be an $X\times Y$ $\parf$-matrix, and let $A'$ be an $X'\times Y'$ $\parf$-matrix such that $X\cup Y = X'\cup Y'$. If $A' \minorof A$ and $A \minorof A'$, then we say that $A$ and $A'$ are \emph{strongly equivalent}. If $\phi(A')$ is strongly equivalent to $A$ for some partial field automorphism $\phi$ (see below for a definition), then we say $A'$ and $A$ are \emph{equivalent}.

\begin{proposition}[{\cite[Proposition 5.1]{SW96}}]\label{prop:hom}Let $\parf_1, \parf_2$ be nontrivial partial fields and let $\phi: \parf_1 \rightarrow \parf_2$ be a homomorphism. Let $A$ be a $\parf_1$-matrix. Then
\begin{enumerate}
\item $\phi(A)$ is a $\parf_2$-matrix.
\item If $A$ is square then $\det(A) = 0$ if and only if $\det(\phi(A)) = 0$.
\end{enumerate}
\end{proposition}

%%%%%%%%%%%%%%%%%%%%%%%%%%%%%%%%%%%%%%%%%%%%%%%%%%%%%%%%%%%%%%%%%%%%%
\subsection{Partial-field matroids}
%%%%%%%%%%%%%%%%%%%%%%%%%%%%%%%%%%%%%%%%%%%%%%%%%%%%%%%%%%%%%%%%%%%%%
Let $A$ be an $r\times E$ $\parf$-matrix of rank $r$. We define the set
\begin{align}
  \mathcal{B}_A := \{B \subseteq E \mid |B| = r, \det(A[r,B])\neq 0\}.
\end{align}
\begin{theorem}[{\cite[Theorem 3.6]{SW96}}]\label{thm:matrixmatroid}$\mathcal{B}_A$ is the set of bases of a matroid.
\end{theorem}
\begin{proof}
  If $\parf$ is trivial then $\mathcal{B}_A = \emptyset$, and the theorem holds. So suppose $\parf = (\ring, \group)$ is nontrivial. If $\ring$ is a field then the theorem follows immediately. Let $I$ be a maximal ideal of $\ring$, and define $\field := \ring/I$. Two results in commutative algebra are that $I$ exists, and that $\field$ is a field. Let $\phi:\ring\rightarrow \field$ be defined by $\phi(p) = p + I$ for all $p \in \ring$. Then $\phi$ is a ring homomorphism, which also gives a partial-field homomorphism. From Proposition~\ref{prop:hom} we have $\mathcal{B}_A = \mathcal{B}_{\phi(A)}$, and since the latter is the set of bases of a matroid the theorem follows.
\end{proof}
We denote this matroid by $M[A] = (E,\mathcal{B}_A)$. Observe that, since matrices are labelled in this paper, the ground set of $M[A]$ is fixed by $A$. If $M$ is a matroid of rank $r$ on ground set $E$ and there exists an $r\times E$ $\parf$-matrix $A$ such that $M = M[A]$, then we say that $M$ is $\parf$-representable.

A matroid is \emph{representable} if there exist a field $\field$ and a matrix $A$ over $\field$ such that $M = M[A]$. The proof of Theorem~\ref{thm:matrixmatroid} shows that every matroid representable over some partial field is also representable over some field. Conversely, every field is also a partial field, so we can equally well say that a matroid is representable if there exist a partial field $\parf$ and a $\parf$-matrix $A$ such that $M = M[A]$.

Partial fields may provide more insight in the representability of a matroid. The following result is also a corollary of Proposition~\ref{prop:hom}.
\begin{corollary}[{\cite[Corollary 5.3]{SW96}}]\label{cor:hom}Let $\parf_1$ and $\parf_2$ be partial fields and let $\phi: \parf_1 \rightarrow \parf_2$ be a nontrivial homomorphism. If $A$ is a $\parf_1$-matrix then $M[\phi(A)] = M[A]$. \end{corollary}
It follows that, if $M$ is a $\parf_1$-representable matroid, then $M$ is also $\parf_2$-representable.

\begin{lemma}[{\cite[Proposition 4.1]{SW96}}]
  Let $A$ be an $r\times E$ $\parf$-matrix, and $B$ a basis of $M[A]$. Then there exists a $\parf$-matrix $A'$ such that $M[A'] = M[A]$ and $A'[r,B]$ is an identity matrix.
\end{lemma}
Now let $A$ be a $B\times (E\setminus B)$ matrix with entries in $\parf$. Let $A'$ be the $B\times E$ matrix $A' = [I\: A]$, where $I$ is a $B\times B$ identity matrix. For all $B' \subseteq E$ with $|B'| = |B|$ we have $\det(A'[B,B']) = \pm \det(A[B\setminus B', (E\setminus B)\cap B'])$. Hence $A'$ is a $\parf$-matrix if and only if $A$ is a $\parf$-matrix. We say that $[I\: A]$ is a \emph{$B$-representation} of $M$ for basis $B$.

It follows that the following function is indeed a rank function for a $\parf$-matrix $A$:
\begin{align}
  \rank(A) := \max \{r \mid A \textrm{ has an } r\times r \textrm{ submatrix } A' \textrm{ with } \det(A') \neq 0\}.
\end{align}

\begin{lemma}\label{lem:matrixmatroid}
  Let $A$ be an $X\times Y$ $\parf$-matrix, and $S\subseteq X, T \subseteq Y$.
  If $M = M[I\: A]$ and $N = M\contract S \delete T$, then $N = M[I'\: A']$, where $A' = A - S - T$.
\end{lemma}
Let $X,Y$ be finite, disjoint sets, let $A_1$ be an $X\times Y$ $\parf_1$-matrix, and let $A_2$ be an $X\times Y$ $\parf_2$-matrix. Let $A := A_1\otimes A_2$ be the $X\times Y$ matrix such that $A_{uv} = ((A_1)_{uv},(A_2)_{uv})$.
\begin{lemma}[{\cite[{Lemma~2.19}]{PZ08lift}}]\label{lem:listoffields}If $A_1$ is a $\parf_1$-matrix, $A_2$ is a $\parf_2$-matrix, and $M[I\: A_1] = M[I\: A_2]$ then $A_1\otimes A_2$ is a $\parf_1\otimes\parf_2$-matrix and $M[I\: A_1\otimes A_2] = M[I\: A_1]$.
\end{lemma}

%%%%%%%%%%%%%%%%%%%%%%%%%%%%%%%%%%%%%%%%%%%%%%%%%%%%%%%%%%%%%%%%%%%%%
\subsection{Cross ratios and fundamental elements}
%%%%%%%%%%%%%%%%%%%%%%%%%%%%%%%%%%%%%%%%%%%%%%%%%%%%%%%%%%%%%%%%%%%%%
Let $A$ be an $X\times Y$ $\parf$-matrix. We define the \emph{cross ratios} of $A$ as the set
\begin{align}
    \crat(A) := \left\{p \mid \left[\begin{smallmatrix}1 & 1\\p & 1\end{smallmatrix}\right] \minorof A \right\}.
\end{align}
\begin{lemma}\label{lem:cratminor}
  If $A' \minorof A$ then $\crat(A') \subseteq \crat(A)$.
\end{lemma}
An element $p \in \parf$ is called \emph{fundamental} if $1-p \in \parf$. We denote the set of fundamental elements of $\parf$ by $\fun(\parf)$.

Suppose $F\subseteq \fun(\parf)$. We define the \emph{associates} of $F$ as
\begin{align}
  \assoc F := \bigcup_{p\in F}\crat\left( \left[\begin{smallmatrix}
    1 & 1\\
    p & 1
  \end{smallmatrix}\right]
  \right).
\end{align}
We have
\begin{proposition}If $p \in \fun(\parf)$ then
  $\assoc\{p\}\subseteq \fun(\parf)$.
\end{proposition}
The following lemma gives a complete description of the structure of $\assoc\{p\}$.
\begin{lemma}\label{lem:assocdesc}
If $p \in \{0,1\}$ then $\assoc\{p\} = \{0,1\}$. If $p \in \fun(\parf)\setminus\{0,1\}$ then
\begin{align}
  \assoc\{p\} = \big\{p, 1-p,\frac{1}{1-p},\frac{p}{p-1},\frac{p-1}{p},\frac{1}{p}\big\}.\label{eq:associates}
\end{align}
\end{lemma}
The key observation for a proof is that a minor of a matrix depends, up to scaling, only on the sequences of row and column indices of the final matrix. For a $2\times 2$ matrix there are 24 choices for these. After scaling these so that the appropriate entries are equal to 1, at most 6 distinct values appear in the bottom left corner.

By Lemma~\ref{lem:cratminor}, $\assoc\{p\} \subseteq \crat(A)$ for every $p \in \crat(A)$.

%%%%%%%%%%%%%%%%%%%%%%%%%%%%%%%%%%%%%%%%%%%%%%%%%%%%%%%%%%%%%%%%%%%%%
\subsection{Normalization}
%%%%%%%%%%%%%%%%%%%%%%%%%%%%%%%%%%%%%%%%%%%%%%%%%%%%%%%%%%%%%%%%%%%%%

Let $M$ be a rank-$r$ matroid with ground set $E$, and let $B$ be a basis of $M$. Let $G(M,B)$ be the bipartite graph with vertices $V(G) = B \cup (E\setminus B)$ and edges $E(G) = \{xy \in B\times (E\setminus B) \mid (B\setminus x) \cup y\in \mathcal{B}\}$. For each $y \in E\setminus B$ there is a unique matroid circuit $C_{B,y} \subseteq B\cup y$, the $B$-\emph{fundamental circuit} of $y$ (see \cite[Corollary 1.2.6]{oxley}). By combining results found, for instance, in \cite[Sections 6.4 and 7.1]{oxley}, it is straightforward to deduce the following:

\begin{lemma}\label{lem:bipconn}
  Let $M$ be a matroid, and $B$ a basis of $M$.
  \begin{enumerate}
    \item $xy \in E(G)$ if and only if $x \in C_{B,y}$.
    \item\label{bip:conn} $M$ is connected if and only if $G(M,B)$ is connected.
    \item\label{bip:threeconn} If $M$ is 3-connected, then $G(M,B)$ is 2-connected.
  \end{enumerate}
\end{lemma}

Let $A$ be an $X\times Y$ matrix, $X\cap Y = \emptyset$. With $A$ we associate a bipartite graph $\bip(A) := (V,E)$, where $V: = X \cup Y$ and let $E := \{xy \in X\times Y \mid A_{xy} \neq 0\}$. The following lemma generalizes a result of \citet{BL76}, which can be found in \citet[Theorem 6.4.7]{oxley}. Recall that $\sim$ denotes scaling-equivalence.

\begin{lemma}\label{lem:bipproperties}
  Let $\parf$ be a partial field and $A$ a $\parf$-matrix. Suppose $M=M[I\: A]$.
  \begin{enumerate}
    \item \label{eq:bipnonzeroentries}$G(M,X) = \bip(A)$.
    \item \label{eq:bipscaling}Let $T$ be a maximal spanning forest of $\bip(A)$ with edges $e_1, \ldots, e_k$. Let $p_1, \ldots, p_k \in \parf^*$. Then there exists a unique matrix $A'\sim A$ such that $A'_{e_i} = p_i$.
  \end{enumerate}
\end{lemma}

Let $A$ be a matrix and $T$ a maximal spanning forest for $\bip(A)$. We say that $A$ is \emph{$T$-normalized} if $A_{xy} = 1$ for all $xy \in T$. By the lemma there is always an $A'\sim A$ that is $T$-normalized. We say that $A$ is \emph{normalized} if it is $T$-normalized for some maximal spanning forest $T$, the \emph{normalizing} spanning forest.

A \emph{walk} in a graph $G = (V,E)$ is a sequence $W = (v_0, \ldots, v_n)$ of vertices such that $v_iv_{i+1} \in E$ for all $i \in \{0, \ldots, n-1\}$. If $v_n = v_0$ and $v_i \neq v_j$ for all $0 \leq i < j < n$ then we say that $W$ is a \emph{cycle}.

\begin{definition}\label{def:sig}
  Let $A$ be an $X\times Y$ matrix with entries in a partial field $\parf$, with $X\cap Y = \emptyset$. The \emph{signature} of $A$ is the function $\sigma_A:(X\times Y) \cup (Y\times X) \rightarrow \parf$ defined by $\sigma_A(vw) = 0$ if $vw$ is not an edge of $\bip(A)$, and by
  \begin{align}
    \sigma_A(vw) := \left\{ \begin{array}{cl} A_{vw}\quad & \textrm{if } v \in X, w \in Y\\
                                             1/A_{vw}\quad & \textrm{if } v \in Y, w \in X
                           \end{array}\right.
  \end{align}
  if $vw$ is an edge of $\bip(A)$. If $C = (v_0, v_1, \ldots, v_{2n-1},v_{2n})$ is a cycle of $\bip(A)$ then we define
  \begin{align}
    \sigma_A(C) := (-1)^{|V(C)|/2}\prod_{i=0}^{2n-1} \sigma_A(v_iv_{i+1}).
  \end{align}
\end{definition}

\begin{lemma}[{\cite[Lemma 2.27]{PZ08lift}}]\label{lem:signature}
  Let $A$ be an $X\times Y$ matrix with entries from a partial field $\parf$.
  \begin{enumerate}
    \item\label{sgn:scale} If $A'\sim A$ then $\sigma_{A'}(C)= \sigma_A(C)$ for all cycles $C$ in $\bip(A)$.
    \item\label{sgn:pivot} Let $C = (v_0, \ldots, v_{2n})$ be an induced cycle of $\bip(A)$ with $v_0 \in X$ and $n \geq 3$. Suppose $A' := A^{v_0v_1}$ is such that all entries are in $\parf$. Then $C' = (v_2, v_3, \ldots, v_{2n-1},v_2)$ is an induced cycle of $\bip(A')$ and $\sigma_{A'}(C') = \sigma_A(C)$.
    \item\label{sgn:det} Let $C = (v_0, \ldots, v_{2n})$ be an induced cycle of $\bip(A)$. If $A'$ is obtained from $A$ by scaling rows and columns such that $A'_{v_iv_{i+1}} = 1$ for all $i > 0$, then $A'_{v_0v_1} = (-1)^{|V(C)|/2}\sigma_A(C)$ and $\det(A[V(C)]) = 1-\sigma_A(C)$.
  \end{enumerate}
\end{lemma}

\begin{corollary}\label{cor:signaturefund}
  Let $A$ be an $X\times Y$ $\parf$-matrix. If $C$ is an induced cycle of $\bip(A)$ then $\sigma_A(C) \in \crat(A) \subseteq \fun(\parf)$.
\end{corollary}

%%%%%%%%%%%%%%%%%%%%%%%%%%%%%%%%%%%%%%%%%%%%%%%%%%%%%%%%%%%%%%%%%%%%%
\subsection{Examples of partial fields}
%%%%%%%%%%%%%%%%%%%%%%%%%%%%%%%%%%%%%%%%%%%%%%%%%%%%%%%%%%%%%%%%%%%%%
The following partial fields were studied in \cite{PZ08lift}. We collected their basic properties in the appendix of this paper.
\begin{description}
  \item[Regular.] $\reg = (\Q, \langle -1\rangle)$;
  \item[Near-regular.] $\nreg = (\Q(\alpha),\langle -1, \alpha, 1-\alpha\rangle)$;
  \item[Dyadic.] $\dyadic = (\Q,\langle -1, 2\rangle)$;
  \item[Sixth-roots-of-unity.] $\psru = (\C,\langle\zeta\rangle)$, where $\zeta$ is a primitive complex sixth root of unity, i.e. a root of $x^2-x+1=0$;
  \item[Golden ratio.] $\golrat = (\R,\langle -1,\tau\rangle)$, where $\tau$ is the \emph{golden ratio}, i.e. a root of $x^2-x-1=0$;
  \item[$k$-Cyclotomic.] $\cyclo_k = (\Q(\alpha),\langle -1,\alpha, \alpha-1,\alpha^2-1, \ldots, \alpha^k -1\rangle)$;
  \item[Gaussian.] $\gauss = (\C,\langle i,1-i\rangle)$;
  \item[Near-regular mod 2.] $\uniform_1^{(2)} = (\GF(2)(\alpha),\langle\alpha, 1-\alpha\rangle)$.
\end{description}

%%%%%%%%%%%%%%%%%%%%%%%%%%%%%%%%%%%%%%%%%%%%%%%%%%%%%%%%%%%%%%%%%%%%%
\subsection{The Lift Theorem}\label{ssec:crat}
%%%%%%%%%%%%%%%%%%%%%%%%%%%%%%%%%%%%%%%%%%%%%%%%%%%%%%%%%%%%%%%%%%%%%
If ${\cal A}$ is a set of matrices then we define
\begin{align}
  \Crat({\cal A}) := \bigcup_{A \in {\cal A}} \Crat(A).
\end{align}
The following is a slight modification of \cite[Definition 5.11]{PZ08lift}; see also \cite[Definition 4.3.1]{vZ09}.
% \begin{definition}\label{def:liftclass}
%   Let $\parf$ be a partial field and ${\cal A}$ a set of $\parf$-matrices.
%   We define the ${\cal A}$-\emph{lift} of $\parf$ as
%   \begin{align}
%      \lift_{\cal A}\parf := (\ring_{\parf}/I_{{\cal A},\parf}, \langle\tilde F_{\parf} \cup -1\rangle),
%   \end{align}
%   where $\tilde F_{\parf} := \{\tilde p \mid p \in \fun(\parf)\}$ is a set of symbols, one for every fundamental element, $\ring_{\parf} := \Z[\tilde F_{\parf}]$ is the free $\Z$-module on $\tilde F_{\parf}$, and $I_{{\cal A},\parf}$ is the ideal generated by the following polynomials in $\ring_{\parf}$:
%   \begin{enumerate}
%     \item\label{lc:one} $\tilde 0 - 0$; $\tilde 1 - 1$;
%     \item\label{lc:two} $\tilde{-1} + 1$ if $-1 \in \fun(\parf)$;
%     \item\label{lc:three} $\tilde p+\tilde q-1$, where $p,q \in \fun(\parf)$, $p + q \doteq 1$;
%     \item\label{lc:four} $\tilde p\tilde q-1$, where $p,q \in \fun(\parf)$, $pq = 1$;
%     \item\label{lc:five} $\tilde p\tilde q\tilde r-1$, where $p,q,r \in \fun(\parf)$, $pqr = 1$, and
%     \begin{align}
%       \begin{bmatrix}1 & 1 & 1\\ 1 & p & q^{-1}
%       \end{bmatrix} \minorof A
%     \end{align}
%     for some $A \in {\cal A}$.
%   \end{enumerate}
% \end{definition}
% \begin{theorem}\cite[{Lemma 5.13}]{PZ08lift}\label{thm:liftclass}
%   Let $\parf$ be a partial field and ${\cal A}$ a set of $\parf$-matrices, and let $M$ be a matroid. If $M = M[I\: A]$ for some $A \in {\cal A}$ then $M$ is $\lift_{\cal A}\parf$-representable.
% \end{theorem}

\begin{definition}\label{def:liftclass}
  Let $\parf$ be a partial field and ${\cal A}$ a set of $\parf$-matrices.
  We define the ${\cal A}$-\emph{lift} of $\parf$ as
  \begin{align}
     \lift_{\cal A}\parf := (\ring_{\cal A}/I_{\cal A}, \langle\tilde F_{\cal A} \cup -1\rangle),
  \end{align}
  where $\tilde F_{\cal A} := \{\tilde p \mid p \in \Crat({\cal A})\}$ is a set of symbols, one for every cross ratio of a matrix in ${\cal A}$, $\ring_{\cal A} := \Z[\tilde F_{\cal A}]$ is the polynomial ring over $\Z$ with indeterminates $\tilde F_{\cal A}$, and $I_{\cal A}$ is the ideal generated by the following polynomials in $\ring_{\cal A}$:
  \begin{enumerate}
    \item\label{lc:one} $\tilde 0 - 0$; $\tilde 1 - 1$;
    \item\label{lc:two} $\tilde{-1} + 1$ if $-1 \in \Crat({\cal A})$;
    \item\label{lc:three} $\tilde p+\tilde q-1$, where $p,q \in \Crat({\cal A})$, $p + q = 1$;
    \item\label{lc:four} $\tilde p\tilde q-1$, where $p,q \in \Crat({\cal A})$, $pq = 1$;
    \item\label{lc:five} $\tilde p\tilde q\tilde r-1$, where $p,q,r \in \Crat({\cal A})$, $pqr = 1$, and
    \begin{align}
      \begin{bmatrix}1 & 1 & 1\\ 1 & p & q^{-1}
      \end{bmatrix} \minorof A
    \end{align}
    for some $A \in {\cal A}$.
  \end{enumerate}
\end{definition}

The following result is, essentially, \cite[Lemma 5.12]{PZ08lift}. However, the changes in the definition above require \cite[Theorem 3.5]{PZ08lift} to be restated in terms of cross ratios for a proof; see \cite[Theorem 4.3.3]{vZ09}.
\begin{theorem}\label{thm:liftclass}
  Let $\parf$ be a partial field, let ${\cal A}$ be a set of $\parf$-matrices, and let $M$ be a matroid. If $M = M[I\: A]$ for some $A \in {\cal A}$ then $M$ is $\lift_{\cal A}\parf$-representable.
\end{theorem}

%%%%%%%%%%%%%%%%%%%%%%%%%%%%%%%%%%%%%%%%%%%%%%%%%%%%%%%%%%%%%%%%%%%%%
\subsection{Sub-partial fields}\label{ssec:subpf}
%%%%%%%%%%%%%%%%%%%%%%%%%%%%%%%%%%%%%%%%%%%%%%%%%%%%%%%%%%%%%%%%%%%%%
$(\ring',\group')$ is a sub-partial field of $(\ring,\group)$ if $\ring'$ is a subring of $\ring$ and $\group'$ is a subgroup of $\group$ with $-1 \in \group'$.
\begin{definition}
  Let $\parf = (\ring, \group)$ be a partial field, and let $S \subseteq \parf^*$. Then
\begin{align}
\parf[S] := (\ring,\langle S\cup \{-1\}\rangle).
\end{align}
\end{definition}
We say that a sub-partial field $(\ring',\group')$ of $(\ring,\group)$ is \emph{induced} if there exists a subring $\ring''\subseteq \ring'$ such that $\group' = \group\cap \ring''$.
If $\parf'$ is an induced sub-partial field of $\parf$ then
\begin{align}\label{eq:funclosed}
  \fun(\parf') = \fun(\parf)\cap \parf'.
\end{align}
Not every sub-partial field is induced. Consider, for example, $\cyclo_2[\alpha, 1-\alpha] \cong U_1$. We have $\alpha^2\in \fun(\cyclo_2)$ and $\alpha^2 \in \uniform_1$, but $\alpha^2 \not\in \fun(\uniform_1)$.

\begin{definition}
  Let $\parf, \parf'$ be partial fields with $\parf'\subseteq \parf$, and let $A$ be a $\parf$-matrix. We say that $A$ is a \emph{scaled $\parf'$-matrix} if $A \sim A'$ for some $\parf'$-matrix $A'$.
\end{definition}
Normalization plays an important role:
\begin{lemma}
  If $A$ is a scaled $\parf'$-matrix and $A$ is normalized, then $A$ is a $\parf'$-matrix.
\end{lemma}
\begin{proof}
  Let $T$ be a normalizing spanning forest for $A$, and let $A'\sim A$ be a $\parf'$-matrix. By Lemma~\ref{lem:bipproperties}\eqref{eq:bipscaling} there exists a $T$-normalized $\parf'$-matrix $A'' \sim A'$. But by the same lemma, $A'' = A$.
\end{proof}
\begin{lemma}\label{lem:entrynotinpprime}
  Let $\parf, \parf'$ be partial fields such that $\parf'$ is an induced sub-partial field of $\parf$. Let $A$ be a $\parf$-matrix such that all entries of $A$ are in $\parf'$. Then $A$ is a $\parf'$-matrix.
\end{lemma}
\begin{proof}
  From \eqref{eq:funclosed} it is straightforward to deduce that, if $p,q \in \parf'$ and $p+q \in \parf$, then $p+q \in \parf'$. Combined with the definition of a pivot and Lemma~\ref{lem:detpivot} the result now follows easily.
\end{proof}
The following theorem will be used in Section~\ref{sec:pfdesign}.
\begin{theorem}\label{thm:crossratpf}
  Let $A$ be an $X\times Y$ $\parf$-matrix. Then $A$ is a scaled $\parf[\crat(A)]$-matrix.
\end{theorem}
\begin{proof}
  Let $A$ be a counterexample with $|X|+|Y|$ minimal, and define $\parf' := \parf[\crat(A)]$. Without loss of generality we assume that $A$ is normalized with normalizing spanning forest $T$.
  \begin{claim}\label{cl:Tnormclaim}
    If every entry of $A$ is in $\parf'$ and $A' \sim A$ is $T'$-normalized for some maximal spanning forest $T'$ then every entry of $A'$ is in $\parf'$.
  \end{claim}
  \begin{subproof}
    We prove this for the case $T' = (T\setminus xy)\cup x'y'$ for edges $xy, x'y'$ with $x,x' \in X$ and $y,y' \in Y$. The claim follows by induction. Without loss of generality assume $T, T'$ are trees. Let $X_1\cup Y_1, X_2\cup Y_2$ be the components of $T\setminus e$ such that $x \in X_1, y \in Y_2$. Let $p := A_{x'y'}$. Then $A'$ is the matrix obtained from $A$ by multiplying all entries in $A[X,Y_2]$ by $p^{-1}$ and all entries in $A[X_2,Y]$ by $p$. Since $p \in \parf'$ the claim follows.
  \end{subproof}
  \begin{claim}Every entry of $A$ is in $\parf'$.
  \end{claim}
  \begin{subproof}
    Suppose this is not the case. Let $H$ be the subgraph of $\bip(A)$ consisting of all edges $x'y'$ such that $A_{x'y'} \in \parf'$. Let $xy$ be an edge of $\bip(A)\setminus H$, i.e. $p := A_{xy} \in \parf\setminus\parf'$. Clearly $1 \in \parf'$, so $T\subseteq H$. Therefore $H$ contains an $x-y$ path $P$. Choose $xy$ and $P$ such that $P$ has minimum length. Then $C := P\cup xy$ is an induced cycle of $\bip(A)$. By Corollary \ref{cor:signaturefund}, $\sigma_A(C) \in \crat(A)$. By Definition \ref{def:sig} we have $\sigma_A(C) = q p$ for some $q \in \parf'$. But then $p = q^{-1} \sigma_A(C) \in \parf'$, a contradiction.
  \end{subproof}
  Suppose $A$ has a square submatrix $A'$ such that $\det(A') \not\in\parf'$. Since $|X|+|Y|$ is minimal and we can extend a maximal spanning forest of $A'$ to a maximal spanning forest of $A$, we have that $A = A'$. Observe that $A$ can not be a $2\times 2$ matrix, since all possible determinants of such matrices are in $\parf'$ by definition. Pick an edge $xy$ such that $A_{xy}\neq 0$. By Claim \ref{cl:Tnormclaim} we may assume that the normalizing spanning forest $T$ of $A$ contains all edges $xy'$ such that $A_{xy'} \neq 0$ and $x'y$ such that $A_{x'y} \neq 0$. Consider $A^{xy}$, the matrix obtained from $A$ by pivoting over $xy$. All entries of this matrix are in $\parf'$. By Lemma~\ref{lem:detpivot} we have $\det(A) = (-1)^{x+y}\det(A^{xy}-\{x,y\})$. The latter is the determinant of a strictly smaller matrix which is, by induction, a $\parf'$-matrix, a contradiction.
\end{proof}
\begin{corollary}
  $A$ is a scaled $\parf'$-matrix if and only if $\crat(A)\subseteq\parf'$.
\end{corollary}
Clearly $\parf[\crat(A)]$ is the smallest partial field $\parf'\subseteq \parf$ such that $A$ is a scaled $\parf'$-matrix. As a corollary we have the following (which was stated without proof as Proposition~5.4 in \cite{PZ08lift}).
\begin{corollary}\label{cor:closureispf}
  If a matroid $M$ is representable over a partial field $\parf$, then $M$ is representable over $\parf[\fun(\parf)]$.
\end{corollary}

%%%%%%%%%%%%%%%%%%%%%%%%%%%%%%%%%%%%%%%%%%%%%%%%%%%%%%%%%%%%%%%%%%%%%
\subsection{Connectivity}
%%%%%%%%%%%%%%%%%%%%%%%%%%%%%%%%%%%%%%%%%%%%%%%%%%%%%%%%%%%%%%%%%%%%%
Let $M$ be a matroid with ground set $E$. For $Z \subseteq E$, define the connectivity function $\lambda_M(Z) := \rank(Z) + \rank(E-Z) - \rank(E)$. A partition of the ground set $(Z_1,Z_2)$ is a \emph{$k$-separation} if $|Z_1|,|Z_2|\geq k$ and $\lambda_M(Z_1) < k$. A $k$-separation is \emph{exact} if $\lambda_M(Z_1) = k-1$. A matroid is \emph{$k$-connected} if it has no $k'$-separation for any $k'<k$, and it is \emph{connected} if it is 2-connected.

We now translate the concept of connectivity into our language of matrices. We say that a matrix $A$ is \emph{$k$-connected} if $M[I\: A]$ is $k$-connected. We define $\lambda_A := \lambda_{M[I\: A]}$. The following lemma gives a characterization of the connectivity function in terms of the ranks of certain submatrices of $A$.
\begin{lemma}[\citet{TruI}]\label{lem:matrixconn}
  Suppose $A$ is an $(X_1\cup X_2)\times(Y_1\cup Y_2)$ $\parf$-matrix (where $X_1, X_2, Y_1, Y_2$ are pairwise disjoint). Then
  \begin{align}
    \lambda_A(X_1\cup Y_1) = \rank(A[X_1,Y_2]) + \rank(A[X_2,Y_1]).
  \end{align}
\end{lemma}
For the proof of the Confinement Theorem we need a more detailed understanding of separations. The following definitions are taken from \citet{GGK}. Our notation is different because we define the concepts only for representation matrices, but it is close to that of \citet{GHW05}. \citet{TruIII} discusses the same concepts, and also gives a very detailed analysis of the structure of the resulting matrices. Let $A$ be an $X\times Y$ $\parf$-matrix, and let $A' := A[E']$ for some $E' \subseteq X\cup Y$. Suppose $(Z_1',Z_2')$ is a $k$-separation of $A'$. We say that this $k$-separation is \emph{induced} in $A$ if there exists a $k$-separation $(Z_1,Z_2)$ of $A$ with $Z_1' \subseteq Z_1$ and $Z_2' \subseteq Z_2$.
\begin{definition}\label{def:blockseq}
A \emph{blocking sequence} for $(Z_1',Z_2')$ is a sequence of elements $v_1, \ldots, v_t$ of $E\setminus E'$ such that
\begin{enumerate}
  \item\label{bls:first} $\lambda_{A[E'\cup v_1]}(Z_1') = k$;
  \item\label{bls:middle} $\lambda_{A[E'\cup \{v_i,v_{i+1}\}]}(Z_1'\cup v_i) = k$ for $i = 1, \ldots, t-1$;
  \item\label{bls:last} $\lambda_{A[E'\cup v_t]}(Z_1'\cup v_t) = k$;
  \item No proper subsequence of $v_1, \ldots, v_t$ satisfies the first three properties.
\end{enumerate}
\end{definition}
We need the following results, which can be found in both \citet{GGK} and \citet{TruIII}:
\begin{lemma}\label{lem:blockseq}
  Let $(Z_1',Z_2')$ be an exact $k$-separation of a submatrix $A[E']$ of $A$. Exactly one of the following holds:
  \begin{enumerate}
    \item There exists a blocking sequence for $(Z_1',Z_2')$;
    \item $(Z_1',Z_2')$ is induced.
  \end{enumerate}
\end{lemma}

\begin{lemma}\label{lem:blockseqalternating}If $v_1, \ldots, v_t$ is a blocking sequence, then $v_i\in X$ implies $v_{i+1} \in Y$ and $v_i \in Y$ implies $v_{i+1} \in X$ for $i = 1, \ldots, t-1$.
\end{lemma}

%%%%%%%%%%%%%%%%%%%%%%%%%%%%%%%%%%%%%%%%%%%%%%%%%%%%%%%%%%%%%%%%%%%%%
\section{The Confinement Theorem}\label{sec:conf}
%%%%%%%%%%%%%%%%%%%%%%%%%%%%%%%%%%%%%%%%%%%%%%%%%%%%%%%%%%%%%%%%%%%%%
\begin{definition}\label{def:confinement}
    Let $\parf,\parf'$ be partial fields with $\parf'\subseteq\parf$, $B$ a $\parf'$-matrix, and $M$ a $\parf$-representable matroid. Then $B$ \emph{confines} $M$ if, for all $\parf$-matrices $A$ such that $M = M[I\: A]$ and $B \minorof A$, $A$ is a scaled $\parf'$-matrix.
\end{definition}
\begin{definition}
  Let $\parf, \parf'$ be partial fields with $\parf'\subseteq\parf$, and $N, M$ matroids such that $N \minorof M$. Then $N$ \emph{confines} $M$ if $B$ confines $M$ for every $\parf'$-matrix $B$ with $N = M[I\: B]$.
\end{definition}
Note that if $B$ confines $M$, then every $\parf'$-matrix $B'$ strongly equivalent to $B$ confines $M$, and $B^T$ confines $M^*$.

The following theorem reduces verifying whether $B$ confines a matroid $M$ to a finite check, provided that $M$ and $B$ are 3-connected and $\parf'$ is induced.
\begin{theorem}[Confinement Theorem]\label{thm:confinement}
  Let $\parf, \parf'$ be partial fields such that
  $\parf' \subseteq \parf$ and $\parf'$ is induced.
  Let $B$ be a 3-connected scaled $\parf'$-matrix. Let $A$ be a 3-connected $\parf$-matrix with $B$ as a submatrix. Then exactly one of the following is true:
  \begin{enumerate}
    \item \label{eq:stablemat} $A$ is a scaled $\parf'$-matrix;
    \item \label{eq:notstablemat} $A$ has a 3-connected minor $A'$ with rows $X'$, columns $Y'$, such that
    \begin{itemize}
      \item $A'$ is not a scaled $\parf'$-matrix.
      \item $B$ is isomorphic to $A'-U$ for some $U$ with $|U\cap X'|\leq 1,|U\cap Y'|\leq 1$;
      \item If $B$ is isomorphic to $A'-\{x,y\}$ then at least one of $A'-x, A'-y$ is 3-connected.
    \end{itemize}
  \end{enumerate}
\end{theorem}
Let $\parf, \parf', B$ be as in Definition~\ref{def:confinement}. If there exists a $p \in \fun(\parf)\setminus\fun(\parf')$, then the 2-sum of $M[I\: B]$ with $U_{2,4}$ will have a representation by a $\parf$-matrix $A$ that has a minor $\left[\begin{smallmatrix}
   1 & 1\\
   p & 1
 \end{smallmatrix}\right]$, and therefore $A$ is not a scaled $\parf'$-matrix. It follows that the 3-connectivity requirements in the theorem are essential. The following technical lemma is used in the proof of Theorem~\ref{thm:confinement} to deal with 2-separations that may crop up in certain minors of $A$.
\begin{lemma}\label{lem:pathshortening}
  Let $\parf, \parf'$ be partial fields such that $\parf'$ is an induced sub-partial field of $\parf$. Let $A$ be a 3-connected $X\times Y$ $\parf$-matrix that has a submatrix $A' = A[V,W]$ such that
  \begin{enumerate}
    \item\label{pre1} $V = X_0 \cup x_1$, $W = Y_0 \cup \{y_1,y_2\}$ for some nonempty $X_0, Y_0$ and $x_1 \in X\setminus X_0, y_1, y_2 \in Y\setminus Y_0$;
    \item\label{pre2} $A[X_0, Y_0\cup\{y_1\}]$ is connected;
    \item\label{pre3} $A[X_0, Y_0\cup \{y_1\}]$ is a scaled $\parf'$-matrix;
    \item\label{pre4} $A'$ is not a scaled $\parf'$-matrix;
    \item\label{pre5} $\lambda_{A'} (X_0 \cup Y_0) = 1$.
  \end{enumerate}
  Then there exists an $\tilde X \times\tilde Y$ $\parf$-matrix $\tilde A$ strongly equivalent to $A$ with a submatrix $\tilde A' = \tilde A[\tilde V, \tilde W]$ such that
  \let\saveenumi\theenumi
  \renewcommand{\theenumi}{\Roman{enumi}}
  \begin{enumerate}
    \item\label{post1} $|\tilde V|= |V|$, $|\tilde W| \leq |W|$;
    \item\label{post2} $X_0 \subset \tilde V$, $Y_0 \subset \tilde W$, and $\tilde A[X_0, Y_0] = A[X_0, Y_0]$;
    \item\label{post3} There exists a $\tilde y_1 \in \tilde W\setminus Y_0$ such that $\tilde A[X_0, \tilde y_1]\cong A[X_0, y_1]$;
    \item\label{post4} $\tilde A'$ is not a scaled $\parf'$-matrix;
    \item\label{post5} $\lambda_{\tilde A'}(X_0\cup Y_0) \geq 2$.
  \end{enumerate}
  \let\theenumi\saveenumi
\end{lemma}
\begin{proof}
  Let $\parf$, $\parf'$, $A$, $X_0$, $Y_0$, $x_1$, $y_1$, $y_2$ be as in the lemma.
  We say that a quadruple $(\tilde A, \tilde x_1,\tilde y_1,\tilde y_2)$ is \emph{bad} if $\tilde A$ is strongly equivalent to $A$, Conditions~\eqref{post1}--\eqref{post4} hold with $\tilde V = X_0\cup \tilde x_1$ and $\tilde W = Y_0\cup \{\tilde y_1,\tilde y_2\}$, but $\lambda_{\tilde A'}(X_0\cup Y_0) =1$. Clearly $(A,x_1,y_1,y_2)$ is a bad quadruple.

  Let $(\tilde A, \tilde x_1, \tilde y_1, \tilde y_2)$ be a bad quadruple. Since $A$ is 3-connected, there exists a blocking sequence for the 2-separation $(X_0 \cup Y_0, \{\tilde x_1, \tilde y_1,\tilde y_2\})$ of $\tilde A[\tilde V \cup \tilde W]$. Suppose $(\tilde A, \tilde x_1, \tilde y_1, \tilde y_2)$ was chosen such that the length of a shortest blocking sequence $v_1, \ldots, v_t$ is as small as possible. Without loss of generality $(\tilde A, \tilde x_1, \tilde y_1, \tilde y_2) = (A,x_1,y_1,y_2)$.

  $A[X_0, y_2]$ cannot consist of only zeroes, because otherwise $A'$ could not be anything other than a scaled $\parf'$-matrix. By scaling we may assume that
  \begin{align}\label{eq:pathshortform}
    A' = \kbordermatrix{
           & Y_0 & y_1 & y_2 \\
       X_0 & A_0 & c & c\\
       x_1 & 0   & 1 & p
    },
  \end{align}
  with $X_0, Y_0$ nonempty, $p \not \in \parf'$, $c_i \in \parf'$ for all $i \in X_0$, and $c_i = 1$ for some $i \in X_0$. We will now analyze the blocking sequence $v_1, \ldots, v_t$.

  \paragraph{Case I.}Suppose $v_t \in X$. By Definition~\ref{def:blockseq}\eqref{bls:last} and Lemma~\ref{lem:matrixconn} we have $\rank(A[X_0\cup v_t, \{y_1,y_2\}]) = 2$. If $A_{v_ty_2} = 0$ then $A_{v_ty_1} \neq 0$. Since $(A, x_1, y_2,y_1)$ is a bad quadruple that also has $v_1, \ldots, v_t$ as blocking sequence, we may assume that $A_{v_ty_2} \neq 0$. Define $r := A_{v_ty_1}$ and $s := A_{v_ty_2}$. Then $r \neq s$.

  Suppose $r/s \not \in \parf'$. If $t > 1$ then $A_{v_ty} = 0$ for all $y \in Y_0$. But then $(A,v_t,y_1,y_2)$ is again a bad quadruple, and $v_1, \ldots, v_{t-1}$ is a blocking sequence for the 2-separation $(X_0 \cup Y_0, \{v_t,y_1,y_2\})$ of $A[X_0\cup v_t, Y_0 \cup \{y_1,y_2\}]$, contradicting our choice of $(A,x_1,y_1,y_2)$. If $t = 1$ then there is some $y \in Y_0$ such that $A_{v_ty} \neq 0$. Let $\tilde A$ be obtained from $A$ by multiplying row $v_t$ with $(A_{v_ty})^{-1}$. Then $A_{v_ty_i} \not\in \parf'$ for exactly one $i \in \{1,2\}$. Then $\tilde A$, $\tilde V := X_0 \cup v_t$, $\tilde W := Y_0 \cup y_i$ satisfy \eqref{post1}--\eqref{post5}.

  Therefore $r/s \in \parf'$. Consider the matrix $\tilde A$ obtained from $A^{x_1y_2}$ by scaling column $y_1$ by $(1-p^{-1})^{-1}$, column $x_1$ by $-p$, and row $y_2$ by $(1-p^{-1})$. Then
  \begin{align}
        \tilde A[X_0 \cup \{v_t,y_2\}, Y_0 \cup \{y_1,x_1\}] = \kbordermatrix{
           & Y_0 & y_1 & x_1 \\
       X_0 & A_0 & c & c\\
       v_t & d   & \frac{rp-s}{p-1} & s\\
       y_2 & 0   & 1 & 1-p
    }.
  \end{align}
  Clearly $(\tilde A, y_2, y_1, x_1)$ is a bad quadruple. Suppose $\frac{rp-s}{p-1} = q \in \parf'$. Then $(q-r)p = q-s$. But this is only possible if $q-r = q-s = 0$, contradicting the fact that $r \neq s$. The set $\{v_1, \ldots, v_{t}\}$ still forms a blocking sequence of this matrix. Hence we can apply the arguments of the previous case and obtain again a shorter blocking sequence.

  \paragraph{Case II.}Suppose $v_t \in Y$. Then $A_{x_1v_t} \neq 0$, again by Definition~\ref{def:blockseq}\eqref{bls:last} and Lemma~\ref{lem:matrixconn}. Suppose all entries of $A[X_0,v_t]$ are zero. Let $\tilde A$ be the matrix obtained from $A^{x_1y_1}$ by multiplying column $y_1$ with $-1$, column $y_2$ with $(1-p)^{-1}$, and row $x_1$ with $-1$. Then $(\tilde A, y_1,x_1,y_2)$ is a bad quadruple, $v_1, \ldots, v_t$ is a blocking sequence, and $\tilde A[X_0,v_t]$ is parallel to $A[X_0,y_1]$. Therefore we may assume that some entry of $A[X_0,v_t]$ is nonzero.

  If $A_{x_1v_t} \in \parf'$ then let $\tilde A$ be the matrix obtained from $A$ by scaling row $x_1$ by $p^{-1}$. Otherwise $\tilde A = A$. Then $(\tilde A, x_1, y_2, y_1)$ is again a bad quadruple, and $v_1, \ldots, v_t$ is still a blocking sequence. Hence we may assume that $A_{x_1v_t} \not \in \parf'$. Suppose $t > 1$. Since $v_1, \ldots, v_{t-1}$ is not a blocking sequence, we must have $A_{v_{t-1}y_1} = A_{v_{t-1}y_2}$. But then $v_1, \ldots, v_{t-1}$ is a blocking sequence for the 2-separation $(X_0 \cup Y_0, \{x_1,y_1,v_t\})$ of $A[X_0 \cup x_1, Y_0 \cup \{y_1,v_t\}]$. But $(A, x_1,y_1,v_t)$ is a bad quadruple, contradicting minimality of $v_1, \ldots, v_t$.

  Hence $t = 1$. But then $\rank(A[X_0, \{v_t,y_1,y_2\}]) = 2$ and therefore $A$, $\tilde V := X_0 \cup x_1$, $\tilde W := Y_0 \cup \{y_1,v_t\}$ satisfy \eqref{post1}--\eqref{post5}.
\end{proof}
Truemper \cite[{Theorem 13.2}]{TruIII} and \citet{GHW05} show that, in the worst case, a minimum blocking sequence for a 2-separation has size 5. The difference between that result and Lemma~\ref{lem:pathshortening} is that in our case the minor we wish to preserve is contained in one side of the separation. This is analogous to what happens in proofs of the Splitter Theorem.

We need three more preliminary results before proving Theorem~\ref{thm:confinement}.
The effect of a pivot over $xy$ is limited to entries having a distance close to that of $x$ and $y$. The following lemma makes this explicit.
\begin{lemma}\label{lem:pivotii}
  Let $A$ be an $X\times Y$ $\parf$-matrix, and let $d$ be the distance function of $\bip(A)$. Let $x \in X$, $y \in Y$ be such that $A_{xy} \neq 0$.
  Let $X' := \{x' \in X \mid d_{\bip(A)}(x',y) >1 \}$ and $Y' := \{y' \in Y \mid d_{\bip(A)}(x,y') > 1\}$. Then $A^{xy}[X',Y\setminus y] = A[X',Y\setminus y]$ and $A^{xy}[X\setminus x, Y'] = A[X\setminus x, Y']$.
\end{lemma}
\begin{proof}
  $A_{xy'} = 0$ whenever $d_{\bip(A)}(x,y') > 1$. Likewise, $A_{x'y} = 0$ whenever $d_{\bip(A)}(x',y) > 1$. The result follows immediately from Definition~\ref{def:pivot}.
\end{proof}
\begin{definition}
  Let $G = (V,E)$ be a connected graph, and let $U \subseteq V$ be such that $G[U]$ is connected. A \emph{$U$-tree} $T$ is a spanning tree for $G$ such that $T$ contains a shortest $v-U$ path for every $v \in V\setminus U$. If $T'$ is a spanning tree of $G[U]$ then $T$ is a \emph{$U$-tree extending $T'$} if $T$ is a $U$-tree and $T'\subseteq T$.
\end{definition}
\begin{lemma}\label{lem:Utree}
  Let $G = (V,E)$ be a connected graph, let $U\subseteq V$, and let $T$ be a $U$-tree for $G$.
  Let $x,y,y' \in V\setminus U$ such that $d_G(U,y) = d_G(U,y') = d_G(U,x) - 1$, $xy \in T$. Then $T' := (T \setminus xy)\cup xy'$ is a $U$-tree.
\end{lemma}
\begin{proof}
  Let $W\subseteq V$ be the set of vertices of the component containing $x$ in $T\setminus xy$. For all $v \in W$, $d_G(U,v) \geq d_G(U,x)$. Therefore $y' \not \in W$ and $T'$ is a spanning tree of $\bip(A)$. Clearly $T'$ contains a shortest $U-x$ path, from which the result follows.
\end{proof}
\begin{lemma}\label{lem:scalecompatible}
  Let $A$ be a connected $X\times Y$ $\parf$-matrix, let $U \subseteq X\cup Y$, and let $T$ be a $U$-tree for $\bip(A)$. Let $x\in X\setminus U$, $y,y' \in Y$ be such that $d_{\bip(A)}(U,y) = d_{\bip(A)}(U,y') = d_{\bip(A)}(U,x) - 1$, $xy \in T$. Let $W$ be the set of vertices of the component containing $x$ in $T\setminus xy$. Suppose $A$ is $T$-normalized. If $A' \sim A$ is $((T\setminus xy)\cup xy')$-normalized, then $A'[X\setminus W, Y\setminus W] = A[X\setminus W, Y\setminus W]$.
\end{lemma}
\begin{proof}
   $A'$ is obtained from $A$ by scaling all rows in $X\cap W$ by $(A_{xy'})^{-1}$ and all columns in $Y\cap W$ by $A_{xy'}$.
\end{proof}
\begin{proof}[Proof of Theorem~\ref{thm:confinement}]
\addtocounter{theorem}{-5}

Let $\parf,\parf'$ be partial fields such that $\parf'$ is an induced sub-partial field, and let $B$ be an $X_0\times Y_0$ $\parf'$-matrix. We may assume that $B$ is normalized, say with spanning tree $T_0$. Note that the theorem holds for $A, B$ if and only if it holds for $A^T, B^T$.
Suppose now that the theorem is false. Then there exists an $X\times Y$ $\parf$-matrix $A$ with the following properties:
\begin{itemize}
  \item $A$ is 3-connected;
  \item $X_0 \subseteq X$, $Y_0 \subseteq Y$, and $B = A[X_0,Y_0]$;
  \item Neither \eqref{eq:stablemat} nor \eqref{eq:notstablemat} holds.
\end{itemize}
We call such a matrix \emph{bad}. The following is clear:
\begin{claim}If $A$ is a bad matrix and $\tilde A$ is strongly equivalent to $A$ such that $\tilde A[X_0,Y_0] = B$, then $\tilde A$ is also bad.
\end{claim}
We say that a triple $(A,T,xy)$ is a \emph{bad triple} if
 \begin{itemize}
   \item $A$ is bad;
   \item $T$ is an $(X_0\cup Y_0)$-tree extending $T_0$;
   \item $A$ is $T$-normalized;
   \item $x\in X$, $y \in Y$, and $A_{xy} \in \parf\setminus\parf'$.
 \end{itemize}
Since we assumed the existence of bad matrices, by Lemma~\ref{lem:entrynotinpprime} bad triples must also exist.

For $v \in X\cup Y$ we define $d_A(v) := d_{\bip(A)}(v,X_0\cup Y_0)$. If $xy$ is an edge of $\bip(A)$ then $d_A(xy) := \max\{d_A(x),d_A(y)\}$. If $xy$ is an edge of $\bip(A)$ then $|d_A(x) -d_A(y)| \leq 1$.
\begin{claim}
  There exists a bad triple $(A,T,xy)$ with $d_A(xy) \leq 1$.
\end{claim}
\begin{subproof}
Let $(A,T,xy)$ be chosen among all bad triples such that $d_A(xy)$ is minimal, and after that such that $|d_A(x)-d_A(y)|$ is maximal. By transposing $A, B$ if necessary we may assume that $d_A(x) \geq d_A(y)$. For $i \geq 1$ we define $X_i := \{x \in X \mid d_A(x) = i\}$ and $Y_i := \{y \in Y \mid d_A(y) = i\}$. We also define $X_i^\leq := X_0 \cup \cdots \cup X_i$ and $Y_i^\leq := Y_0 \cup \cdots \cup Y_i$. Suppose $d_A(xy)> 1$. We distinguish two cases.
\paragraph{Case I.} Suppose $d_A(x)=d_A(y)=i$. If $X_{i-1} = \emptyset$ then $Y_{i}=\emptyset$, contradicting our choice of $y$. Since $A$ is normalized, $A_{xy'} = 1$ for some $y' \in Y_{i-1}$, and $A_{x'y} = 1$ for some $x' \in X_{i-1}$. Let $p := A_{xy}$ and $q := A_{x'y'}$. Then $q\in \parf'$.

Let $\tilde A$ be the matrix obtained from $A^{xy}$ by multiplying row $y$ with $p$ and column $x$ with $-p$.

Let $\tilde T$ be an $(X_0\cup Y_0)$-tree extending $T_0$ in $\bip(A^{xy})$, such that $uv \in \tilde T$ for all $uv \in T[(X\setminus x) \cup Y_{i-2}^\leq]$ and all $uv \in T[X_{i-2}^\leq \cup (Y\setminus y)]$. By Lemma~\ref{lem:pivotii} such a tree exists. Let $\tilde A \sim A^{xy}$ be $\tilde T$-normalized. By Lemma~\ref{lem:pivotii} and Lemma~\ref{lem:scalecompatible}, $\tilde A_{x'y'} = (A^{xy})_{x'y'}$. But $\tilde A_{x'y'} = q - p^{-1} \not \in \parf'$, so $(\tilde A, \tilde T, x'y')$ is a bad triple with $d_{\tilde A}(x'y') = i-1 < i$, a contradiction.

\paragraph{Case II.} Suppose $d_A(x) = i+1$, $d_A(y) = i$.
Since $A$ is normalized, $A_{xy'} = 1$ for some $y'\in Y$ with $d_A(y') = i$. If $\rank(A[X_i^\leq, \{y',y\}]) = 1$ then we apply Lemma~\ref{lem:pathshortening} with $A' = A[X_i^\leq\cup x, Y_{i-1}^\leq \cup \{ y' , y\}]$. If $|\tilde W| < |W|$ then $\tilde A[\tilde x_1, Y_0]$ has some nonzero entry. But then $(\tilde A, \tilde T, \tilde x_1\tilde y_1)$ would be a bad triple for some $(X_0\cup Y_0)$-tree $\tilde T$ with $d_{\tilde A}(\tilde x_1\tilde y_1) \leq i$, a contradiction. Therefore $\tilde W = Y_0 \cup \{\tilde y_1, \tilde y_2\}$ for some $\tilde y_1, \tilde y_2$, and $\rank(\tilde A[X_0, \{\tilde y_1, \tilde y_2\}]) = 2$. Now $\tilde A[X_0, \tilde W]$ must be a scaled $\parf'$-matrix, since $d_{\tilde A}(v) \leq i$ for all $v \in X_0 \cup \tilde W$.

It follows that we may assume that $(A,T,xy)$ were chosen such that $xy' \in T$ and $\rank(A[X_i^\leq, \{y',y\}]) = 2$.
Suppose there exists an $x_1 \in X_i^\leq$ with $d_A(x_1) = i-1$ such that $A_{x_1y} \neq 0$ and $A_{x_1y'} \neq 0$.
Again by Lemma~\ref{lem:scalecompatible} we may assume that $x_1y, x_1y' \in T$. Since
\begin{align}
  \rank(A[X'',\{y',y\}]) = 2,
\end{align}
there is a row $x_2\in X_i^\leq$ such that
\begin{align}
  A[\{x_1,x_2,x\},\{y',y\}] = \begin{bmatrix}1 & 1\\
                                                    r & s\\
                                                    1 & p
  \end{bmatrix}
\end{align}
with $r \neq s$ and $p \in \parf\setminus\parf'$. Consider $A^{xy}$. By Lemma~\ref{lem:pivotii} we have $d_{A^{xy}}(x_1) = i-1$ and $d_{A^{xy}}(y') = i$. By the same lemma, there is a spanning tree $T'$ of $\bip(A^{xy})$ with $yy', x_1y',x_1x \in T'$ and, for all $u \in X\setminus x$ and $v\in Y$ with $d_{A^{xy}}(v) \leq i-1$, $uv \in T'$ if and only if $uv \in T$. Let $A' \sim A^{xy}$ be $T'$-normalized. Then
\begin{align}
  A'[\{x_1,x_2,y\},\{y',x\}] = \begin{bmatrix}1 & 1\\
                                                    \frac{pr-s}{p-1} & s\\
                                                    1 & 1-p
  \end{bmatrix}
\end{align}
But $\frac{rp-s}{p-1} \not \in \parf'$. Therefore $(A',T', x_2y')$ is a bad triple, and $d_{A'}(x_2y') \leq i$, contradicting our choice of $(A,T,xy)$.
Therefore we cannot find an $x_1$ such that $A_{x_1y'} \neq 0$ and $A_{x_1y} \neq 0$. But in that case there exist $x_1,x_2$ with $d_A(x_1) = d_A(x_2) = i-1$ and $A_{x_1y'} \neq 0$, $A_{x_2y} \neq 0$. Again we may assume without loss of generality that $x_1y', x_2y, xy' \in T$. Then
\begin{align}
   A[\{x_1,x_2,x\},\{y',y\}] = \begin{bmatrix}1 & 0\\
                                                    0 & 1\\
                                                    1 & p
  \end{bmatrix}.
\end{align}
Again, consider $A^{xy}$. By Lemma~\ref{lem:pivotii} we have $d_{A^{xy}}(x_1) = d_{A^{xy}}(x_2) = i-1$ and $d_{A^{xy}}(y') = i$. By the same lemma, there is a spanning tree $T'$ of $\bip(A^{xy})$ with $yy', x_1y',x_2x \in T'$ and, for all $u \in X\setminus x$ and $v\in Y$ with $d_{A^{xy}}(v) \leq i-1$, $uv \in T'$ if and only if $uv \in T$. Let $A' \sim A^{xy}$ be $T'$-normalized. Then
\begin{align}
  A'[\{x_1,x_2,y\},\{y',x\}] = \begin{bmatrix}1 & 0\\
                                                    -p^{-1} & 1\\
                                                    1 & -p
  \end{bmatrix}
\end{align}
But then $(A',T', x_2y')$ is a bad triple, and $d_{A'}(x_2y') \leq i$, again contradicting our choice of $(A,T,xy)$.
\end{subproof}

Let $(A,T,xy)$ be a bad triple with $d_A(xy) = 1$.
\begin{claim}
  $d_A(x) = d_A(y) = 1$.
\end{claim}
\begin{subproof}
Suppose that $x \in X_0, y \in Y_1$. Let $A' := A[X_0, Y_0\cup y]$. Then $A'[X_0,y]$ contains a 1, since $y$ is at distance 1 from $B$ therefore spanned by $T_1$. It also contains an entry equal to $p$, so it has at least two nonzero entries and cannot be a multiple of a column of $B$. It follows that $A'$ satisfies the conditions of Case~\eqref{eq:notstablemat} of the theorem, a contradiction.
\end{subproof}
Therefore $x \in X_1, y \in Y_1$. Consider the submatrix $A' := A[X_0 \cup x, Y_0 \cup y]$. Row $A_{xy_0} = 1$ for some  $y_0 \in Y_0$, $A_{x_0y} = 1$ for some $x_0 \in X_0$. Define $b := A[X_0,y]$ and $c := A[x,Y_0]$.
\begin{claim}
  Without loss of generality, $b$ is parallel to $A[X_0,y_0]$ for some $y_0 \in Y_0$ and $c$ is a unit vector with $A_{xy_0} = 1$.
\end{claim}
\begin{subproof}
  If $b$ is not a unit vector and not parallel to a column of $B$, then $A'$ satisfies all conditions of Case~\eqref{eq:notstablemat}, a contradiction. If both $b$ and $c$ are unit vectors, and $c$ is such that $A_{xy_0} = 1$, then $A^{xy_0}[X_0,(Y_0\setminus y_0 \cup x)\cup y]$ satisfies all conditions of Case~\eqref{eq:notstablemat}, a contradiction.

  By transposing $A, B$ if necessary we may assume that $b$ is parallel to some column $y'$ of $B$. We scale column $y$ so that the entries of $b$ are equal to those of $A[X_0,y']$. If $c$ has a nonzero in a column $y_0 \neq y'$, then the matrix $A[X_0, Y_0\setminus y' \cup y]$ is isomorphic to $B$, and the matrix $A'' := A[X_0\cup x, (Y_0\setminus y') \cup y]$ satisfies all conditions of \eqref{eq:notstablemat}, a contradiction.
\end{subproof}
Now we apply Lemma~\ref{lem:pathshortening} with $A' = A[X_0\cup x,Y_0 \cup y]$, where $y_1 = y_0$ and $y_2 = y$. But the resulting minor $\tilde A$ satisfies all conditions of Case~\eqref{eq:notstablemat}, a contradiction.
\end{proof}
\addtocounter{theorem}{5}

Whittle's Stabilizer Theorem~\cite{Whi96b} is an easy corollary of the Confinement Theorem.
\begin{definition}
  Let $\parf$ be a partial field, and $N$ a 3-connected $\parf$-representable matroid on ground set $X'\cup Y'$, where $X'$ is a basis.
  Let $M$ be a 3-connected matroid on ground set $X\cup Y$ with minor $N$, such that $X$ is a basis of $M$, $X'\subseteq X$, and $Y'\subseteq Y$.
  Let $A_1, A_2$ be $X\times Y$ $\parf$-matrices such that $M = M[I\: A_1] = M[I\: A_2]$.
  Then $N$ is a \emph{$\parf$-stabilizer for $M$} if $A_1[X',Y'] \sim A_2[X',Y']$ implies $A_1 \sim A_2$ for all choices of $A_1, A_2$.
\end{definition}
\begin{theorem}[Stabilizer Theorem]\label{cor:stabilizer}
  Let $\parf$ be a partial field, and $N$ a 3-connected $\parf$-representable matroid. Let $M$ be a 3-connected $\parf$-representable matroid having an $N$-minor. Then at least one of the following is true:
  \begin{enumerate}
    \item \label{eq:stablemat2} $N$ stabilizes $M$;
    \item \label{eq:notstablemat2} $M$ has a 3-connected minor $M'$ such that
    \begin{itemize}
      \item $N$ does not stabilize $M'$;
      \item $N$ is isomorphic to $M'\contract x$, $M'\delete y$, or $M'\contract x \delete y$, for some $x,y \in E(M')$;
      \item If $N$ is isomorphic to $M'\contract x \delete y$ then at least one of $M'\contract x, M'\delete y$ is 3-connected.
    \end{itemize}
  \end{enumerate}
\end{theorem}
\begin{proof}
  Consider the product partial field $\parf_0 := \parf \otimes \parf$, and define $\parf_0' := \{(p,p) \mid p \in \parf\}$. Then $\parf_0'$ is an induced sub-partial field of $\parf_0$. Apply Theorem~\ref{thm:confinement} to all matrices $A, B$ such that $M = M[I\: A]$, $N = M[I\: B]$, $B\minorof A$, $A$ is a $\parf_0$-matrix, and $B$ is a $\parf_0'$-matrix.
\end{proof}

%%%%%%%%%%%%%%%%%%%%%%%%%%%%%%%%%%%%%%%%%%%%%%%%%%%%%%%%%%%%%%%%%%%%%
\section{The universal partial field of a matroid}\label{sec:pfdesign}
%%%%%%%%%%%%%%%%%%%%%%%%%%%%%%%%%%%%%%%%%%%%%%%%%%%%%%%%%%%%%%%%%%%%%
%%%%%%%%%%%%%%%%%%%%%%%%%%%%%%%%%%%%%%%%%%%%%%%%%%%%%%%%%%%%%%%%%%%%%
\subsection{The bracket ring}
%%%%%%%%%%%%%%%%%%%%%%%%%%%%%%%%%%%%%%%%%%%%%%%%%%%%%%%%%%%%%%%%%%%%%
In this section we find the ``most general'' partial field over which a single matroid is representable.
Our construction is based on the bracket ring from \citet{Wte75a}.
Let $M = (E,{\cal B})$ be a rank-$r$ matroid.  For every $r$-tuple $Z \in E^r$ we introduce a symbol $[Z]$, the ``bracket'' of $Z$, and a symbol $\bar{[Z]}$. Suppose $Z = (x_1, \ldots, x_r)$. Define $\{Z\} := \{x_1, \ldots, x_r\}$, and $Z/x\rightarrow y$ as the $r$-tuple obtained from $Z$ by replacing each occurrence of $x$ by $y$. We define
\begin{align}
  {\cal Z}_M := \{[Z] \mid Z \in E^r\}\cup\{\bar{[Z]} \mid \{Z\} \textrm{ is a basis of } M\}.
\end{align}
\begin{definition}\label{def:bracketring}
$I_M$ is the ideal in $\Z[{\cal Z}_M]$ generated by the following polynomials:
\begin{enumerate}
  \item\label{bra:dependent}$[Z]$, for all $Z$ such that $\{Z\} \not \in {\cal B}$;
  \item \label{bra:antisymmetry}$[Z] - \sign(\sigma)[Z^\sigma]$, for all $Z$ and all permutations $\sigma:\{1,\ldots,r\}\rightarrow\{1,\ldots,r\}$;
  \item \label{bra:syzygies}$[x_1,x_2,U][y_1,y_2,U] - [y_1,x_2,U][x_1,y_2,U] - [y_2, x_2, U][y_1,x_1,U]$, for all $x_1,x_2,y_1,y_2 \in E$ and $U \in E^{r-2}$;
  \item \label{bra:inverse}$[Z] \bar{[Z]} - 1$, for all $Z$ such that $\{Z\} \in {\cal B}$;
  for all $Z \in E^r$.
\end{enumerate}
\end{definition}
Now we define
\begin{align}
  \bracketring_M := \Z[{\cal Z}_M]/I_M.
\end{align}
Relations \eqref{bra:dependent}--\eqref{bra:syzygies}
are the same as those in White's construction~\cite{Wte75a}. They accomplish that the brackets behave like determinants in $\bracketring_M$. A special case of \eqref{bra:dependent} occurs when $|\{Z\}| <r$. In that case $Z$ must have repeated elements. Relations~\eqref{bra:inverse} are not present in the work of White.
\begin{lemma}\label{lem:bracketringhom}
  Let $\parf = (\ring, \group)$ be a partial field and $A$ an $r\times E$ $\parf$-matrix such that $M = M[A]$. Then there exists a ring homomorphism $\phi:\bracketring_M\rightarrow\ring$.
\end{lemma}
\begin{proof}
  Let $\phi':\Z[{\cal Z}_M]\rightarrow\field$ be determined by $\phi'([Z]) = \det(A[r,Z])$ and $\phi'(\bar{[Z]}) = \det(A[r,Z])^{-1}$. We show that $I_M \subseteq\ker(\phi')$, from which the result follows. Relations \eqref{bra:dependent} follow from linear dependence, Relations \eqref{bra:antisymmetry} from antisymmetry, and Relations~\eqref{bra:syzygies} from the \emph{3-term Grassmann-Pl\"ucker relations} (see, for example, \citet[{Page~127}]{oriented}).
\end{proof}

With our addition to White's construction we are actually able to represent $M$ over the partial field $(\bracketring_M, \langle {\cal Z}_M\cup\{-1\}\rangle)$. Note that, as soon as $\rank(M)\geq 2$, we can pick a basis $Z$ and an odd permutation $\sigma$ of the elements of $Z$ to obtain $[Z^\sigma]\bar{[Z]} = -1 \in \langle {\cal Z}_M\rangle$, making the $-1$ in the definition of the partial field redundant.
\begin{definition}
  Let $M$ be a rank-$r$ matroid. Let $B \in E^r$ be such that $\{B\}$ is a basis of $M$. Then $A_{M,B}$ is the $B\times (E\setminus B)$ matrix with entries in $\bracketring_M$ given by
  \begin{align}
    (A_{M,B})_{uv} := [B/u \rightarrow v]/[B].
  \end{align}
\end{definition}
\begin{lemma}\label{lem:Brepmat}$A_{M,B}$ is a $(\bracketring_M, \bracketring_M^*)$-matrix.
\end{lemma}
\begin{proof}
Let $A := A_{M,B}$. Let $x\in B, y \in E\setminus B$ be such that $B' := (B\setminus x) \cup y$ is again a basis. We study the effect of a pivot over $xy$. Let $u \in \{B\}\setminus x, v \in (E\setminus \{B\})\setminus y$. We have
\begin{align}
  (A^{xy})_{yx} & = A_{xy}^{-1} = [B]/[B/x\rightarrow y],\\
  (A^{xy})_{yv} & = A_{xy}^{-1} A_{xv} = ([B]/[B/x\rightarrow y])([B/x\rightarrow v]/[B])\notag\\
                & = [B'/y\rightarrow v]/[B/x\rightarrow y],\\
  (A^{xy})_{ux} & = -A_{xy}^{-1} A_{uy} = -([B]/[B/x\rightarrow y])([B/u\rightarrow y]/[B])\notag\\
                & = [(B/x\rightarrow y)/u\rightarrow x]/[B/x\rightarrow y],\label{eq:thirdpivotbracket}\\
  (A^{xy})_{uv} & = A_{uv} - A_{xy}^{-1} A_{uy} A_{xv} \notag\\
                & = \frac{[B/u\rightarrow v]}{[B]} - \frac{[B]}{[B/x\rightarrow y]}\frac{[B/u\rightarrow y]}{[B]} \frac{[B/x\rightarrow v]}{[B]}\notag\\
                & = \frac{[B/x\rightarrow y][B/u\rightarrow v]-[B/u\rightarrow y][B/x\rightarrow v]}{[B][B/x\rightarrow y]}\notag\\
                & = \frac{[B][(B/x\rightarrow y)/u\rightarrow v]}{[B][B/x\rightarrow y]}.\label{eq:fourthpivotbracket}
\end{align}
For \eqref{eq:thirdpivotbracket} we note that $[(B/x\rightarrow y)/u\rightarrow x]$ is a permutation of $[B/u\rightarrow y]$; by \ref{def:bracketring}\eqref{bra:antisymmetry} the minus sign vanishes. For \eqref{eq:fourthpivotbracket} we use \ref{def:bracketring}\eqref{bra:syzygies}. In short, for every entry $u \in B', v \in (E\setminus B')$ we have
\begin{align}
  (A^{xy})_{uv} = [B'/u\rightarrow v]/[B'],
\end{align}
so $(A_{M,B})^{xy} = A_{M,B'}$. By Lemma~\ref{lem:detpivot} we find that every subdeterminant is equal to $\prod_{i=1}^k[Z_i]/[B_i]$ for some $Z_i,B_i\in E^r$ with all $\{B_i\}$ bases, and therefore, by \ref{def:bracketring}\eqref{bra:inverse}, every subdeterminant is either equal to zero or invertible. The lemma follows.
\end{proof}
\begin{lemma}\label{lem:matrixrepresentsM}
  Let $M$ be a matroid such that $\bracketring_M$ is nontrivial. If $B$ is a basis of $M$ then $M = M[I\: A_{M,B}]$.
\end{lemma}
\begin{proof}
  Clearly $M$ and $M[I\: A_{M,B}]$ have the same set of bases.
\end{proof}
The following theorem gives a characterization of representability:
\begin{theorem}\label{lem:branontriv}
  $M$ is representable if and only if $\bracketring_M$ is nontrivial.
\end{theorem}
\begin{proof}
  $\phi(1) = 1$ for any homomorphism $\phi$. Therefore, if $M$ is representable then Lemma~\ref{lem:bracketringhom} implies that $\bracketring_M$ is nontrivial. Conversely, if $\bracketring_M$ is nontrivial then Lemma~\ref{lem:matrixrepresentsM} shows that $M$ is representable over the partial field $(\bracketring_M, \bracketring_M^*)$.
\end{proof}
The following lemma can be proven by adapting the proof of the corresponding result in \citet[{Theorem~8.1}]{Wte75a}:
\begin{lemma}
  $\bracketring_{M^*} \cong \bracketring_M$.
\end{lemma}
Finally we consider the effect of taking a minor.
\begin{definition}
  Let $M = (E,{\cal B})$ be a matroid, and let $U,V\subseteq E$ be disjoint ordered subsets such that $U$ is independent and $V$ coindependent.
  Then we define
  \begin{align}
    \tilde\phi_{M,U,V} : \bracketring_{M\contract U \delete V} \rightarrow \bracketring_M,
  \end{align}
  by $\tilde\phi_{M,U,V}([Z]) := [Z\,\, U]$ for all $Z \in (E\setminus(U\cup V))^{r-|U|}$.
\end{definition}
Note that, in a slight misuse of notation, we have written $M\contract U\delete V$ instead of $M\contract\{U\}\delete \{V\}$.
\begin{lemma}\label{lem:bracketringminor}
  $\tilde\phi_{M,U,V}$ is a ring homomorphism.
\end{lemma}
\begin{proof}
  Let $\tilde\phi':\Z[{\cal Z}_{M\contract U \delete V}]\rightarrow\bracketring_M$ be determined by $\tilde\phi'([Z]) := [Z\,\, U]$. It is easy to see that $I_{M\contract U \delete V} \subseteq \ker(\tilde\phi')$. The result follows.
\end{proof}
%Note that the corresponding theorem in White~\cite[Theorem~8.2]{Wte75a} is incorrect: it states that $\bracketring_{M\contract U \delete V} \subseteq \bracketring_M$. A counterexample is obtained by taking $M$ to be the Fano matroid. Then $\bracketring_M$ has characteristic 2, whereas $\bracketring_{M\delete e}$ has characteristic 0.

%%%%%%%%%%%%%%%%%%%%%%%%%%%%%%%%%%%%%%%%%%%%%%%%%%%%%%%%%%%%%%%%%%%%%
\subsection{The universal partial field}
%%%%%%%%%%%%%%%%%%%%%%%%%%%%%%%%%%%%%%%%%%%%%%%%%%%%%%%%%%%%%%%%%%%%%
In principle Theorem~\ref{lem:branontriv} gives a way to compute whether a matroid is representable: all one needs to do is to test whether $1 \in I_M$, which can be achieved by computing a Groebner basis over the integers for $I_M$ (see \citet{BV03} for details). However, for practical computations the partial field $(\bracketring_M,\bracketring_M^*)$ is somewhat unwieldy. In this subsection we rectify this problem.

If $M$ is a matroid then we define the set of \emph{cross ratios of $M$} as
  \begin{align}
    \crat(M) := \crat(A_{M,B}).
  \end{align}
Note that $\crat(M)$ does not depend on the choice of $B$. We introduce the following subring of $\bracketring_M$:
  \begin{align}
    \ring_M := \Z[\crat(M)].
  \end{align}
Now we define the \emph{universal partial field of $M$} as
\begin{align}
  \parf_M := (\ring_M,\langle\crat(M)\cup\{-1\}\rangle).
\end{align}

By Theorem~\ref{thm:crossratpf} we have that, if $M$ is representable, then $M$ is representable over $\parf_M$.
We give an alternative construction of this partial field. Let $M = (E,{\cal B})$ be a rank-$r$ matroid on a ground set $E$, let $B\in{\cal B}$, and let $T$ be a maximal spanning forest for $\bip(M,B)$. For every $x \in B, y \in E\setminus B$ we introduce a symbol $a_{xy}$. For every $B' \in {\cal B}$ we introduce a symbol $i_{B'}$. We define
\begin{align}\label{eq:calY}
  {\cal Y}_M := \{a_{xy} \mid x \in B, y \in E\setminus B\}\cup\{i_{B'} \mid B' \in {\cal B}\}.
\end{align}

Let $\hat A_{M,B}$ be the $B\times (E\setminus B)$ matrix with entries $a_{xy}$.

\begin{definition}\label{def:bracketring2}
  $I_{M,B,T}$ is the ideal in $\Z[{\cal Y}_M]$ generated by the following polynomials:
  \begin{enumerate}
      \item $\det(\hat A_{M,B}[B\setminus Z, (E\setminus B)\cap Z])$ if $Z \not\in {\cal B}$;
      \item $\det(\hat A_{M,B}[B\setminus Z, (E\setminus B)\cap Z]) i_Z-1$ if $Z \in {\cal B}$;
      \item $a_{xy} - 1$ if $xy \in T$;
  \end{enumerate}
  for all $Z \subseteq E$ with $|Z| = r$.
\end{definition}
Now we define
\begin{align}
  \bracketring_{M,B,T} := \Z[{\cal Y}_M]/I_{M,B,T}
\end{align}
and
\begin{align}
  \parf_{M,B,T} := (\bracketring_{M,B,T}, \langle \{i_{B'} \mid B' \in {\cal B}\} \cup -1\rangle).
\end{align}
Finally, $\hat A_{M,B,T}$ is the matrix $\hat A_{M,B}$, viewed as a matrix over $\parf_{M,B,T}$.

The construction of $\parf_{M,B,T}$ is essentially the same as the construction in \citet{Fen84}. The difference between his construction and ours is that we ensure that the determinant corresponding to every basis is invertible.
The proof of Lemma~\ref{lem:bracketringhom} can be adapted to prove the following lemma.
\begin{lemma}\label{lem:pfuniversal}Let $\parf = (\ring, \group)$, and let $M = M[I\: A]$ for some $B\times (E\setminus B)$ $\parf$-matrix $A$ that is $T$-normalized for a maximal spanning forest $T$ of $\bip(A)$.
Then there exists a ring homomorphism $\phi:\bracketring_{M,B,T}\rightarrow\ring$ such that $\phi(\hat A_{M,B,T})= A$.
\end{lemma}

Since two normalized representations of a matroid are equivalent if and only if they are equal, the following is an immediate consequence of this lemma:

\begin{corollary}\label{cor:universalbijection}
  There is a bijection between nonequivalent representations of a matroid $M$ over a partial field $\parf$ and partial-field homomorphisms $\parf_M \rightarrow \parf$.
\end{corollary}

Now we can prove that the two constructions described yield isomorphic partial fields.

\begin{theorem}\label{thm:universalequalsconst}$\bracketring_{M,B,T} \cong \ring_M$ and $\parf_{M,B,T} \cong \parf_M$.
\end{theorem}
\begin{proof}
  Let $A_{M,B,T}$ be the unique $T$-normalized matrix with $A_{M,B,T}\sim A_{M,B}$. By Theorem~\ref{thm:crossratpf} $A_{M,B,T}$ is a $\parf_M$-matrix. By Lemma~\ref{lem:pfuniversal} there exists a homomorphism $\phi:\bracketring_{M,B,T}\rightarrow\ring_M$ such that $\phi(\hat A_{M,B,T}) = A_{M,B,T}$. By Lemma~\ref{lem:bracketringhom} there exists a homomorphism $\psi':\bracketring_M\rightarrow\bracketring_{M,B,T}$ such that $\psi'(A_{M,B}) = \hat A_{M,B,T}$. Note that also $\psi'(A_{M,B,T}) = \hat A_{M,B,T}$. Let $\psi := \psi'|_{\ring_M}$. Now $\phi$ and $\psi$ are both surjective and $\phi(\psi(A_{M,B})) = A_{M,B}$, so that we have $\phi(\psi(p)) = p$ for all $p\in\crat(M)$. Since $\ring_M$ is generated by $\crat(M)$, the result follows.
\end{proof}
We say that a partial field $\parf$ is \emph{universal} if $\parf = \parf_M$ for some matroid $M$. The next lemma, which has a straightforward proof, gives a good reason to study universal partial fields.
\begin{lemma}\label{lem:universalimplieshom}
  Let $\parf$ be a universal partial field, and let ${\cal M}$ be the class of $\parf$-representable matroids. Then all $M \in {\cal M}$ are $\parf'$-representable if and only if there exists a homomorphism $\phi:\parf\rightarrow\parf'$.
\end{lemma}
We conclude this subsection by studying the effect of taking a minor on the universal partial field. The proof of the following lemma is straightforward.
\begin{definition}\label{def:canonichom}
  Let $M = (E,{\cal B})$ be a matroid, and $U,V\subseteq E$ disjoint ordered subsets such that $U$ is independent and $V$ coindependent. Then we define $\phi_{M,U,V}$ as the restriction of $\tilde\phi_{M,U,V}$ to $\Z[\crat(M\contract U\delete V)]$.
\end{definition}
\begin{lemma}\label{lem:universalminorhom}
   $\phi_{M,U,V}$ is a ring homomorphism $\ring_{M\contract U \delete V} \rightarrow \ring_M$.
\end{lemma}
Note that, because of the restriction to cross ratios, $\phi_{M,U,V}$ does not depend on the particular ordering of $U$ and $V$. The function $\phi_{M,U,V}$ is the \emph{canonical homomorphism} $\ring_{M\contract U\delete V}\rightarrow\ring_M$ and induces a partial field homomorphism $\parf_{M\contract U \delete V}\rightarrow\parf_M$.
\begin{lemma}\label{lem:universalhomsubmat}
  Let $M = (E,{\cal B})$ be a matroid, and $U,V\subseteq E$ disjoint subsets such that $U$ is independent and $V$ coindependent. Let $B \in {\cal B}$ be such that $U \subseteq B$, and let $T$ be a maximal spanning forest for $G(M,B)$ extending a maximal spanning forest $T'$ for $G(M\contract U \delete V, B\delete U)$. Then
  \begin{align}
    \phi_{M,U,V}(A_{M\contract U \delete V, B\setminus U, T'}) = A_{M,B,T} - U - V.
  \end{align}
\end{lemma}

%%%%%%%%%%%%%%%%%%%%%%%%%%%%%%%%%%%%%%%%%%%%%%%%%%%%%%%%%%%%%%%%%%%%%
\subsection{Examples}\label{ssec:universal}
%%%%%%%%%%%%%%%%%%%%%%%%%%%%%%%%%%%%%%%%%%%%%%%%%%%%%%%%%%%%%%%%%%%%%
In this section we will see that several well-known partial fields are universal. %Consider the ring $\Q(\alpha_1,\ldots,\alpha_k)$. Define the set
% \begin{align}
%   U_k := \{ x - y \mid x,y \in \{0,1,\alpha_1,\ldots,\alpha_k\}, x \neq y\}.
% \end{align}
% We define the \emph{$k$-uniform} partial field as
% \begin{align}
%   \uniform_k := (\Q(\alpha_1, \ldots, \alpha_k),\langle U_k\rangle).
% \end{align}
% This partial field appears in \citet{Sem97} as the \emph{$k$-regular partial field}. For $k = 0, 1$ it coincides with the regular and near-regular partial fields defined previously. The following theorem and its proof appear, in essence, also in \citet{Fen84}.
% \begin{theorem}$\parf_{U_{2,k+3}} \cong \uniform_k$.
% \end{theorem}
% \begin{proof}
%   Suppose $E(U_{2,k+3}) = \{1,2,\ldots, k+3\}$. Let $B := \{1,2\}$, and $T := \{23\} \cup \{1j \mid j \in \{3,\ldots,k+3\}\}$. Then
%   \begin{align}
%     \hat A_{U_{2,k+3},B,T} = \kbordermatrix{
%           & 3 & 4 & \cdots & k+3\\
%        1  & 1 & 1 & \cdots & 1\\
%        2  & 1 & \alpha_1 & \cdots & \alpha_k
%     }
%   \end{align}
%   where $\alpha_i := a_{2,i+3}$. Let $\alpha_0 := 1$. For $3 \leq i < j \leq k+3$ we have $\{i,j\} \in {\cal B}$.
%   Hence $\det(A_{U_{2,k+3},B,T}[\{1,2\},\{i,j\}]) = \alpha_{j-3} - \alpha_{i-3}$ is invertible.
%   The result follows.
% \end{proof}
Consider the \emph{dyadic} partial field
\begin{align}
  \dyadic := (\Z[\tfrac{1}{2}], \langle -1, 2\rangle).
\end{align}
Consider also the matroid $P_8$, which is a rank-4 matroid with no 3-element dependent sets and exactly ten 4-element dependent sets, indicated by the ten planes in Figure \ref{fig:P8}.

\begin{figure}[tbp]
  \begin{center}
  \includegraphics{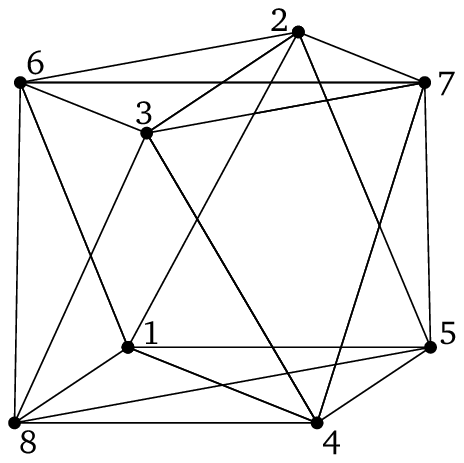}
  \end{center}
  \caption{Geometric representation of the matroid $P_8$}\label{fig:P8}
\end{figure}

\begin{theorem}\label{thm:P8}
  The dyadic partial field $\dyadic$ is the universal partial field of $P_8$.
\end{theorem}
\begin{proof}[Sketch of proof]
  Pick $B = \{1,2,3,4\}$ as basis, and a spanning tree $T$ of $G(M,B)$ with edges $16,17,25,27,28,38,47$. Then
  \begin{align}
    \hat A_{P_8,B,T} = \kbordermatrix{ & 5 & 6 & 7 & 8\\
                                1 & 0 & 1 & 1 & x\\
                                2 & 1 & 0 & 1 & 1\\
                                3 & y & z & 0 & 1\\
                                4 & u & v & 1 & 0},
  \end{align}
  where $x = a_{18}$, $y = a_{35}$, and so on (the $a_{ij}$ are as in \eqref{eq:calY}). Since $\{1,4,5,8\}$ is dependent, it follows that
  \begin{align}
    \det(A[\{2,3\},\{5,8\}]) = 1 - y = 0,
  \end{align}
  so $y = 1$ in $\bracketring_{P_8,B,T}$. Since $\{2,3,6,7\}$ is dependent, it follows that $v = 1$. From the dependency of, respectively, $\{4,6,7,8\}$, $\{1,5,6,7\}$, $\{2,5,6,8\}$, and $\{3,5,7,8\}$ we deduce
  \begin{align}
    1 + (1-x) & = 0\\
    z + (1 - z u) & = 0\\
    u + x(1-z u) & = 0\\
    u + x(1-u) & = 0.
  \end{align}
  Hence $x = 2$. Substituting this in the fourth equation gives $u = 2$. Substituting that in the second equation gives $z = 1$. Note that the third equation is also satisfied this way.

  To complete the proof we should verify that, for all $\left(\!\begin{smallmatrix}8\\4\end{smallmatrix}\!\right)-10$ bases $B'$, $i_{B'}$ is in $\Z[\frac{1}{2}]$. This is equivalent to the fact that all subdeterminants of
  \begin{align}
    \kbordermatrix{ & 5 & 6 & 7 & 8\\
                                1 & 0 & 1 & 1 & 2\\
                                2 & 1 & 0 & 1 & 1\\
                                3 & 1 & 1 & 0 & 1\\
                                4 & 2 & 1 & 1 & 0}
  \end{align}
  are powers of $2$. We leave out this routine but somewhat lengthy check.
\end{proof}

Next we describe, for each $q$, a rank-3 matroid on $3q+1$ elements for which the universal partial field is $\GF(q)$.
For $q$ a prime power, let $Q_q$ be the rank-3 matroid consisting of three distinct $(q+1)$-point lines $L_1, L_2, L_3 \subset \PG(2,q)$ such that $L_1\cap L_2\cap L_3 = \emptyset$. Then $Q_q = Q_3(\GF(q)^*)$, the rank-3 Dowling geometry for the multiplicative group of $\GF(q)$. Now $Q_q^+$ is the matroid obtained from $Q_q$ by adding a point $e \in \PG(2,q)\setminus(L_1\cup L_2\cup L_3)$. For instance, $\deltadot_2 \cong F_7$.
\begin{theorem}
  $\parf_{\deltadot_q} \cong \GF(q)$.
\end{theorem}
\begin{proof}
  Let $\{e_1\} = L_2\cap L_3$, $\{e_2\} = L_1 \cap L_3$, and $\{e_3\} = L_1 \cap L_2$. Then $B := \{e_1,e_2,e_3\}$ is a basis of $\deltadot_q$. If $\alpha$ is a generator of $\GF(q)^*$ then a $B$-representation of $\deltadot_q$ is the following:
  \begin{align}
    A =
    \kbordermatrix{ & e & a_0 & a_1 & & a_{q-2} & b_0 & & b_{q-2} & c_0 & & c_{q-2}\\
                e_1 & 1 &  0  &  0  & \cdots &  0      &  1  & \cdots &  1      &  1  &\cdots &  1\\
                e_2 & 1 &  1  &  1  & &  1      &  0  & &  0      &  1  & &  \alpha^{q-2}\\
                e_3 & 1 &  1  &  \alpha  & \cdots &  \alpha^{q-2}      &  1  &\cdots &  \alpha^{q-2}      &  0  &\cdots &  0
    }.
  \end{align}
  Let $T$ be the spanning tree of $\bip(A)$ with edges $e_1x,e_2x,e_3x$ and, for all $i \in \{0,\ldots,q-2\}$, edges $e_2a_i, e_1b_i, e_1c_i$. Then
  \begin{align}
    \hat A_{\deltadot_q,B,T} =
    \kbordermatrix{ & e & a_0 & a_1 & & a_{q-2} & b_0 & & b_{q-2} & c_0 & & c_{q-2}\\
                e_1 & 1 &  0  &  0  &\cdots &  0      &  1  &\cdots &  1      &  1  &\cdots &  1\\
                e_2 & 1 &  1  &  1  & &  1      &  0  & &  0      &  z_0  & &  z_{q-2}\\
                e_3 & 1 &  x_0  &  x_1  &\cdots &  x_{q-2}      &  y_0  &\cdots &  y_{q-2}      &  0  &\cdots &  0
    }.
  \end{align}
  \begin{claim}$x_0 = y_0 = z_0 = 1$.
  \end{claim}
  \begin{subproof}
    $\det(A[B\setminus e_1,\{e,a_0\}]) = 0$, so $\det(\hat A_{\deltadot_q,B,T}[B\setminus e_1,\{e,a_0\}]) = x_0 - 1 = 0$. Similarly $y_0 =1$ and $z_0 = 1$.
  \end{subproof}
  \begin{claim}
    If $\alpha^k = -1$ then $x_k = y_k = z_k = -1$.
  \end{claim}
  \begin{subproof}
    $\det(A[B,\{a_0,b_0,c_k\}]) = \det(\alpha^k + 1) = 0$, so\\ $\det(\hat A_{\deltadot_q,B,T}[B,\{a_0,b_0,c_k\}]) = z_k + 1 = 0$ and $z_k = -1$. Similarly $x_k = -1$ and $y_k = -1$.
  \end{subproof}
  \begin{claim}
    $x_l = y_l = z_l$ for all $l$.
  \end{claim}
  \begin{subproof}
    Let $k$ be such that $x_k = -1$. Note that, if $\GF(q)$ has characteristic $2$, then $k = 0$. Then $\det(A[B,\{a_k, b_l, c_l\}]) = 0$, so
    \begin{align}
      \det(\hat A_{\deltadot_q,B,T}[B,\{a_k,b_l,c_l\}]) = y_l - z_l = 0.
    \end{align}
    Therefore $y_l = z_l$. Similarly $y_l = x_l$.
  \end{subproof}
  By replacing $a_k$ by $a_0$ in the previous subproof we obtain
  \begin{claim}
    If $\alpha^m = - \alpha^l$ then $x_m = - x_l$, for all $k,l$.
  \end{claim}
  Now we establish the multiplicative structure of $\GF(q)$:
  \begin{claim}If $\alpha^k \alpha^l = \alpha^m$ then $x_k x_l = x_m$.
  \end{claim}
  \begin{subproof}
    Let $n$ be such that $\alpha^m = - \alpha^n$. Then
    $\det(A[B,\{a_k, b_n, c_l\}]) = \alpha^k \alpha^l + \alpha^n = 0$, so $\det(\hat A_{\deltadot_q,B,T}[B,\{a_k,b_n,c_l\}]) = x_k x_l + x_n = 0$, so $x_k x_l = x_m$.
  \end{subproof}
  Finally we establish the additive structure.
  \begin{claim}If $\alpha^k = \alpha^l + 1$ then $x_k = x_l + 1$.
  \end{claim}
  \begin{subproof}
    Let $m$ be such that $\alpha^m = - \alpha^l$. Then $\det(A[B,\{e,a_k,b_m\}]) = \alpha^k + \alpha^m - 1 = 0$, so $\det(\hat A_{\deltadot_q,B,T}[B,\{e,a_k,b_m\}]) = x_k + x_m - 1 = 0$, so $x_k = x_l + 1$.
  \end{subproof}
  This completes the proof.
\end{proof}
We made no attempt to find a smallest matroid with $\GF(q)$ as universal partial field. For $q$ prime it is known that fewer elements suffice: one may restrict the line $L_3$ to $e_2, e_3$, and the point collinear with $e_1$ and $e$.  \citet{Bry82} showed that yet more points may be omitted. \citet{Laz58} described, for primes $p$, a rank-$(p+1)$ matroid with characteristic set $\{p\}$.

\begin{table}[tp]
  \begin{center}
  \begin{tabular}{llllllll}
    \toprule
    $\parf_M\phantom{X}$ & $\GF(q)$ & $\uniform_k$ & $\dyadic$ & $\psru$  \\
    \addlinespace
    $M$ & $\deltadot_q$ & $U_{2,k+3}$ & $F_7^{-}$ & $\AG(2,3)$ \\
    \midrule
    $\parf_M$  & $\golmean$ & $\cyclo_2$ & $\uniform_{1}^{(2)}$ & $\gauss$\\
    \addlinespace
    $M$  & $M[I\: A_1]$ & $M[I\: A_2]$ & $M[I\: A_3]$ & $Q_5$\\
    \bottomrule
  \end{tabular}
  \end{center}
  \caption{Some universal partial fields.}\label{tab:universal}
\end{table}

Without proof we give Table~\ref{tab:universal}, which states that many partial fields that we have encountered so far are indeed universal. In this table we have
\begin{align}
  A_1 &:= \begin{bmatrix}
    1 & 0 & 1 & 1\\
    1 & 1 & \tau^{-1} & \tau\\
    0 & 1 & -1 & \tau\\
    1 & -\tau^{-1} & \tau^{-1} & 0
  \end{bmatrix},\label{eq:royleGR}\\
   A_2 &:= \begin{bmatrix}
    -1 & 0 & 1 & 1\\
    1 & -1 & 0 & \alpha\\
    0 & 1 & -1 & -1
  \end{bmatrix},\\
  A_3 &:= \begin{bmatrix}
    1 & 0 & 1 & 1 & 1\\
    1 & 1 & 0 & 1 & \alpha\\
    0 & 1 & 1 & 1 & 1
  \end{bmatrix},
\end{align}
where $A_1$ is a $\golmean$-matrix, $A_2$ is a $\cyclo_2$-matrix, and $A_3$ is a $\uniform_1^{(2)}$-matrix.

%%%%%%%%%%%%%%%%%%%%%%%%%%%%%%%%%%%%%%%%%%%%%%%%%%%%%%%%%%%%%%%%%%%%%
\subsection{The Settlement Theorem}\label{ssec:settlement}
%%%%%%%%%%%%%%%%%%%%%%%%%%%%%%%%%%%%%%%%%%%%%%%%%%%%%%%%%%%%%%%%%%%%%
The following theorem is a close relative of a theorem on totally free expansions of matroids from \citet{GOVW02} (Theorem 2.2).
\begin{definition}
  Let $M, N$ be matroids such that $N \cong M\contract U \delete V$ with $U$ independent and $V$ coindependent, and let $\phi_{M,U,V}:\ring_N\rightarrow\ring_M$ be the canonical ring homomorphism from Definition \ref{def:canonichom}. Then $N$ \emph{settles} $M$ if $\phi_{M,U,V}$ is surjective.
\end{definition}

\begin{theorem}\label{thm:settlement}
  Let $M, N$ be matroids such that $N = M\contract U \delete V$ with $U$ independent and $V$ coindependent. Exactly one of the following is true:
  \begin{enumerate}
    \item $N$ settles $M$;
    \item $M$ has a 3-connected minor $M'$ such that
    \begin{itemize}
      \item $N$ does not settle $M'$;
      \item $N$ is isomorphic to $M'\contract x$, $M'\delete y$, or $M'\contract x \delete y$ for some $x,y \in E(M')$;
      \item If $N$ is isomorphic to $M'\contract x \delete y$ then at least one of $M'\contract x$ and $M'\delete y$ is 3-connected;
    \end{itemize}
  \end{enumerate}
\end{theorem}
\begin{proof}
  Let $\parf := \parf_M = (\ring_M, \langle \crat(M)\cup \{-1\}\rangle)$. Let
  \begin{align}
    \parf' := (\phi_{M,U,V}(\ring_N), \langle\crat(M)\cup\{-1\}\rangle\cap \phi_{M,U,V}(\ring_N)).
  \end{align}
  Then $\parf'$ is an induced sub-partial field of $\parf$. Let $B$ be a basis of $M$ with $U\subseteq B$ and $T$ be a maximal spanning forest of $G(M,B)$ extending a maximal spanning forest $T'$ of $G(N,B\setminus U)$. Define $B := \phi_{M,U,V}(A_{N,B\setminus U,T'})$ and $A := A_{M,B,T}$. By Lemma~\ref{lem:universalhomsubmat} $B\minorof A$.
  The theorem follows if we apply the Confinement Theorem to $\parf'$, $\parf$, $B$, and $A$.
\end{proof}
Like the theory of totally free expansions, Theorem~\ref{thm:settlement} can be used to show that certain classes of matroids have a bounded number of inequivalent representations. We will use the following lemma to prove such a result in Subsection \ref{ssec:ternary}.

\begin{lemma}\label{lem:U24settlesGF3}Suppose $M$ is a ternary, nonbinary matroid. Then $U_{2,4}$ settles $M$.
\end{lemma}

\begin{proof}
  Since $M$ is nonbinary, $U_{2,4} \minorof M$. No 3-connected 1-element extension or coextension of $U_{2,4}$ is a minor of $M$. Hence $U_{2,4}$ settles $M$. %Let $B$ be a basis of $M$ such that $U\subseteq B$, $V \subseteq E\setminus B$, and $M\contract U \delete V = U_{2,4}$. %Let $\{e_1, e_2, e_3,e_4\}$ be the elements of $U_{2,4}$, with $e_1,e_2 \in B$, and let $T$ be a spanning tree for $\bip(A_{M,B})$ containing $e_1e_3, e_1e_4, e_2e_4$. Suppose $A_1, \ldots, A_k$ are inequivalent, $T$-normalized $B$-representations of $M$. Then there exist homomorphisms $\phi_i:\parf_M\rightarrow\parf$ such that $\phi_i(\hat A_{M,B,T}) = A_i$. But for each $i$, $\phi_i$ is determined uniquely by the image of
%   \begin{align}
%     \hat A_{M,B,T}[\{e_1,e_2\},\{e_3,e_4\}] = \begin{bmatrix}
%       1 & 1\\
%       p & 1
%     \end{bmatrix}.
%   \end{align}
%   Clearly $\phi_i(p) \in \fun(\parf)\setminus \{0,1\}$. The result follows.
\end{proof}

%%%%%%%%%%%%%%%%%%%%%%%%%%%%%%%%%%%%%%%%%%%%%%%%%%%%%%%%%%%%%%%%%%%%%
\section{Applications}\label{sec:applications}
%%%%%%%%%%%%%%%%%%%%%%%%%%%%%%%%%%%%%%%%%%%%%%%%%%%%%%%%%%%%%%%%%%%%%
\subsection{Ternary matroids}\label{ssec:ternary}
We will combine the Lift Theorem, in particular Theorem~\ref{thm:liftclass}, with the Confinement Theorem to give a new proof of the following result by Whittle:
\begin{theorem}[{\citet{Whi97}}]\label{thm:classification}
  Let $M$ be a 3-connected matroid that is representable over $\GF(3)$ and some field that is not of characteristic 3. Then $M$ is representable over at least one of the partial fields $\reg, \nreg, \psru, \dyadic$.
\end{theorem}

\begin{proof}
  Let $\field$ be a field that is not of characteristic 3, and define $\parf := \GF(3)\otimes\field$. Define ${\cal A}$ as the set of all $\parf$-matrices. An $\field$-representable matroid $M$ is ternary if and only if $M = M[I\: A]$ for some $A \in {\cal A}$. We study $\parf' := \lift_{\cal A}\parf$. Since $-1$ is the only nontrivial fundamental element of $\GF(3)$, and $(-1)^3 \neq 1$, we have that $I_{{\cal A},\parf}$, as in Definition~\ref{def:liftclass}, is generated by relations \eqref{lc:one}--\eqref{lc:four}. Recall that the definition of associates of a fundamental element from Lemma \ref{lem:assocdesc}. Consider the set $C := \{\assoc\{\tilde p\} \subseteq \parf' \mid \tilde p \in \tilde F_{\cal A}\}$. Each relation of types \eqref{lc:three},\eqref{lc:four} implies that two elements of $\tilde F_{\cal A}$ are equal. This results either in the identification of two members of $C$, or in a relation \emph{within} one set of associates. Recall that $\parf[S]$ denotes the sub-partial field of $\parf$ generated by $S$ and $-1$.

  \begin{claim}\label{cl:classification}If $\tilde p \in \tilde F_{\cal A}$ then $\parf'[\assoc\{\tilde p\}]$ is isomorphic to one of $\reg,\nreg,\dyadic, \psru$.
  \end{claim}
  \begin{subproof}
    If $\tilde p \in \{0,1\}$ then clearly $\parf'[\assoc\{\tilde p\}] \cong\reg$, so assume $\tilde p \neq 0,1$.
    Consider $\ring := \Z[p_1, \ldots, p_6]$. For each $D \subseteq \{(i,j)\mid i,j \in \{1, \ldots, 6\}, i \neq j\}$, let $I_D$ be the ideal generated by
    \begin{itemize}
      \item $p_i + p_{i+1} - 1$, for $i = 1, 3, 5$;
      \item $p_i p_{i+1} - 1$, for $i = 2, 4, 6$ (where indices are interpreted modulo 6);
      \item $p_i - p_j$, for all $(i,j) \in D$.
    \end{itemize}
    By the discussion above, $\parf'[\assoc\{\tilde p\}] \cong (\ring/I_D, \langle p_1, \ldots, p_6\rangle)$ for some $D$. There are only finitely many sets $D$, so the claim can be proven by a finite check.

    If $D = \emptyset$ then $\parf'[\assoc\{\tilde p\}] \cong \uniform_1$.

    If $|D| \geq 1$ then we may assume, through relabeling, that $(1,j) \in D$ for some $j \in \{2,\ldots,6\}$.

    \paragraph{Case I:} $(1,2) \in D$. Then $p_1 + p_2 - 1 = 2 p_1 - 1 \in I_D$. Therefore $p_6(2p_1 - 1) = 2 - p_6 \in I_D$. Also, $p_5 + p_6 - 1 = p_5 + 1 \in I_D$, and $p_4p_5 - 1 = -p_4 - 1 \in I_D$. Finally, $p_6(p_2p_3-1) = p_3-p_6 \in I_D$. Therefore $p_6 = p_3 = 2$, $p_1 = p_2 = 2^{-1}$, and $p_4 = p_5 = -1$  in $\ring/I_D$.

    Suppose there is another relation that does not follow from the above. Then either $p_1 = p_3$, or $p_1 = p_4$, or $p_3 = p_4$. In the first case, $p_1p_6 - 1 = p_6^2-1 = 4-1 = 3 \in I_D$, so $\ring/I_D$ has characteristic 3, which we assumed was not so. In the second case, $p_1p_6-1 = (-1) 2 -1 = -3 \in I_D$, and again $\ring/I_D$ has characteristic 3. In the third case $p_3 + p_4 -1 = 2 p_3 - 1 = 4 - 1 = 3 \in I_D$, again a contradiction.

    Hence $\ring/I_D \cong \Z[1/2]$.

    \paragraph{Case II:} $(1,3) \in D$. Then $p_3(p_1 + p_2 - 1) = p_1^2 - p_1 + 1 \in I_D$. Also, $p_6(p_2p_3 - 1) = p_2-p_6 \in I_D$, and $p_1(p_3+p_4-1) = p_1^2 + p_1 p_4 - p_1 = p_1p_4 - 1 \in I_D$, so $p_6(p_1p_4 - 1) = p_4 - p_6 \in I_D$. Also, $p_5(p_1p_4-1) = p_1 - p_5 \in I_D$, so $p_1 = p_3 = p_5$ and $p_2 = p_4 = p_6$.

    If there is another relation then $(1,2) \in D$, which was covered in Case I.

    \paragraph{Case III:} $(1,4) \in D$. Then $p_1(p_5+p_6-1) = p_4p_5 + p_1p_6 - p_1 = 2 - p_1 \in I_D$. Therefore $p_1 = p_4 = 2$, and $p_5 = p_6 = 2^{-1}$ in $\ring/I_D$. Now $p_3 + p_4 - 1 = p_3 + 1 \in I_D$, so $p_3 = p_2 = -1$. After relabeling we are back in Case I.

    \paragraph{Case IV:} $(1,5) \in D$. Then $p_6(p_4p_5 - 1) = p_4-p_6 \in I_D$. Moreover, $p_3 + p_4 - 1 = p_3-p_5 \in I_D$. But then $(1,3) \in D$, which was dealt with in Case II.

    \paragraph{Case V:} $(1,6) \in D$. Then $p_3(p_1p_6-1) = p_3(p_1^2-1) = p_3(p_1-1)(p_1+1) = -p_3p_2(p_1+1) = p_1 + 1 \in I_D$, so $p_1 = -1$ in $\ring/I_D$. Hence $p_6 = -1$ as well, and then $p_1 + p_2 -1 = p_2-2 \in I_D$, so $p_2 = 2$. But then $p_3 = 2^{-1}$. Likewise, $p_5 = 2$ and $p_4 = 2^{-1}$. After relabeling we are back in Case I.

%     If $|D| = 1$ then we may assume $D = \{(1,j)\}$ for some $j \in \{2, \ldots, 6\}$. Elementary manipulations of the ideal show that if $j \in \{2,4,6\}$ then $\ring/I_D \cong \Z[1/2]$, whereas for $j \in \{3,5\}$, $\ring/I_D \cong \Z[\zeta]$, where $\zeta$ is a root of $x^2-x+1$.
%     We show this for one case, leaving the remaining cases out. Assume $j = 6$. Then $p_1(p_2p_3 -1) = p_1((1-p_1)p_3-1) = p_1 p_3 - p_1^2 p_3 - p_1 = p_1 p_3 - p_3 - p_1 \in I_D$, since $p_1^2 = p_1 p_6 = 1$ in $\ring/I_D$.  Substituting $p_1+p_3$ for $p_1 p_3$ in $(1-p_1)p_3 - 1$ yields $-p_1 - 1 \in I$, so $p_1 = p_6 = -1$, and the result follows easily.
%
%     If $|D| = 2$ then we may assume $D = \{(1,j), (i,j')\}$ for some $i \in \{1, \ldots, 6\}$ and $j,j' \in\{2,\ldots,6\}$. Note that $\ring/I_D \cong \ring/I_{\{(1,j)\}}/I_{\{(i,j')\}}$. Checks similar to the previous case show that always $\ring/I_D \cong \Z[1/2]$ or $\ring/I_D \cong \Z[\zeta]$ or $\ring/I_D \cong \GF(3)$. The latter can never occur since we assumed that the $\parf'$-representable matroids are also representable over a field that does not have characteristic 3. In the other cases no new relations are implied, so $I_D = I_{\{(1,j)\}}$. Again we leave out the details.
%
%     It follows that no new rings arise for $|D| \geq 2$, and the proof is complete.
  \end{subproof}

  \begin{claim}
    \label{cl:splittable}Suppose $2 \in \parf'$. Then each of the following matrices confines all $\parf$-representable matroids to $\dyadic$:
    \begin{align}
    \begin{bmatrix}
      1 & 1\\
      2 & 1
    \end{bmatrix},
    \begin{bmatrix}
      1 & 1\\
      1/2 & 1
    \end{bmatrix},
    \begin{bmatrix}
      1 & 1\\
      -1 & 1
    \end{bmatrix}.
  \end{align}
  \end{claim}
  \begin{subproof}
    Observe that, since there is no $U_{2,5}$-minor in $\GF(3)$, there exist no ternary 3-connected 1-element extensions or coextensions of these matrices. Hence the claim must hold by the Confinement Theorem.
  \end{subproof}
  We immediately have
  \begin{claim}Let $A \in {\cal A}$ be 3-connected such that $2 \in \Crat(A)$. Then $A$ is a scaled $\dyadic$-matrix.
  \end{claim}
  We now solve the remaining case.
  \begin{claim}Let $A \in {\cal A}$ be 3-connected such that $2 \not\in \Crat(A)$. Then $A$ is a scaled $\reg$-matrix or a scaled $\nreg$-matrix or a scaled $\psru$-matrix.
  \end{claim}
  \begin{subproof}
    Without loss of generality assume that $A$ is normalized. Clearly $2 \not \in \parf'[\Crat(A)]$. Suppose there exists a $\tilde p \in \Crat(A)\setminus\{0,1\}$. Define the sub-partial field $\parf'' := \parf'[ \assoc\{\tilde p\}]$. Since all additive relations are restricted to just one set of associates, we have
    \begin{align}\fun(\parf'') = \fun(\parf'[\Crat(A)])\cap \parf''.
    \end{align}
    By the Confinement Theorem, then, we have that $\left[\begin{smallmatrix}1 & 1\\ p & 1\end{smallmatrix}\right]$ confines all $\parf'[\Crat(A)]$-representable matroids to $\parf''$. The result follows by Claim~\ref{cl:classification}.

    Finally, if $\Crat(A) = \{0,1\}$ then define $\parf'' := \parf'[\emptyset]$. Clearly $\parf''\cong\reg$, and the proof of the claim is complete.
  \end{subproof}
  This completes the proof of the theorem.
\end{proof}

We can also deduce some more information about the number of representations of ternary matroids over other partial fields. We start with a lemma. %two lemmas.
%
% \begin{lemma}\label{lem:splittable}Let $\parf$ be a partial field, and let $M$ be a 3-connected ternary matroid such that $M = M[I\: A]$ for some $\parf$-matrix $A$. Let $p \in \Crat(A)$. Then
% \begin{align}
%   \begin{bmatrix}
%     1 & 1\\
%     1 & p
%   \end{bmatrix}
% \end{align}
%   is a $\parf$-stabilizer for $M$.
% \end{lemma}
%
% \begin{proof}
%   Observe that, since $U_{2,5}$ is not ternary, there exist no 3-connected 1-element extensions or coextensions of these matrices. Hence the lemma must hold.
% \end{proof}

Define the following matrices over $\Q$:
\begin{align}
   A_7 := \begin{bmatrix}
     1 & 1 & 0 & 1\\
     1 & 0 & 1 & 1\\
     0 & 1 & 1 & 1
   \end{bmatrix}, \qquad &
   A_8 := \begin{bmatrix}
     0 & 1 & 1 & 2\\
     1 & 0 & 1 & 1\\
     1 & 1 & 0 & 1\\
     2 & 1 & 1 & 0
   \end{bmatrix},
\end{align}
and define the matroids $F_7^- := M[I\: A_7]$, $P_8 := M[I\: A_8]$.
\begin{lemma}
  The following statements hold for $M \in \{F_7^-,(F_7^-)^*, P_8\}$:
  \begin{enumerate}
%    \item $M$ is a stabilizer for $\dyadic$;
    \item $\dyadic$ is a universal partial field for $M$;
    \item $M$ is uniquely representable over $\dyadic$.
  \end{enumerate}
\end{lemma}
\begin{proof}
  The first statement was proven in Theorem \ref{thm:P8} for $P_8$; a similar argument proves the other cases.
  For the second statement, observe that there is exactly one homomorphism $\parf_M\rightarrow \dyadic$, since $\parf_M = \dyadic$. Hence, by Corollary \ref{cor:universalbijection}, there is exactly one representation of $M$ over $\dyadic$.
\end{proof}

The next theorem strengthens a result by Whittle \cite{Whi96}:

\begin{theorem}\label{thm:regnregdyadic}
  Let $M$ be a 3-connected matroid representable over a partial field $\parf$. Then $M$ has at most $|\fun(\parf)|-2$ inequivalent representations over $\parf$. Moreover, the following hold.
  \begin{enumerate}
    \item If $M$ is regular then $M$ is uniquely representable over $\parf$.
    \item If $M$ is near-regular but not regular then $M$ has exactly $|\fun(\parf)|-2$ representations over $\parf$.
    \item If $M$ is dyadic but not near-regular then $M$ is uniquely representable over $\parf$.
  \end{enumerate}
\end{theorem}
\begin{proof}
  If $M$ is binary then it is well-known that $M$ is uniquely representable over any field. The proof generalizes to partial fields. So we may suppose that $M$ has a $U_{2,4}$-minor. In that case $U_{2,4}$ settles $M$. The universal partial field of $U_{2,4}$ is $\nreg = (\Q(\alpha), \langle -1, \alpha, 1-\alpha\rangle)$. For every $p \in \assoc\{\alpha\}$, there is an automorphism $\phi:\nreg\rightarrow\nreg$ such that $\phi(\alpha) = p$. Since a homomorphism maps fundamental elements to fundamental elements, no other automorphisms exist. It follows that $U_{2,4}$ is uniquely representable over $\nreg$. %By Lemma \ref{lem:U24settlesGF3}, the homomorphism $\nreg\rightarrow\parf_M$ is surjective. Hence any $\parf$-matrix $A$ representing $M$ is uniquely determined by the image of $U_{2,4}$.
  Let $B$ be a basis of $M$ such that $U\subseteq B$, $V \subseteq E\setminus B$, and $M\contract U \delete V = U_{2,4}$. Let $\{e_1, e_2, e_3,e_4\}$ be the elements of $U_{2,4}$, with $e_1,e_2 \in B$, and let $T$ be a spanning tree for $\bip(A_{M,B})$ containing $e_1e_3, e_1e_4, e_2e_4$. Suppose $A_1, \ldots, A_k$ are inequivalent, $T$-normalized $B$-representations of $M$. Then there exist homomorphisms $\phi_i:\parf_M\rightarrow\parf$ such that $\phi_i(\hat A_{M,B,T}) = A_i$. But for each $i$, $\phi_i$ is determined uniquely by the image of
  \begin{align}
    \hat A_{M,B,T}[\{e_1,e_2\},\{e_3,e_4\}] = \begin{bmatrix}
      1 & 1\\
      p & 1
    \end{bmatrix}.
  \end{align}
  Clearly $\phi_i(p) \in \fun(\parf)\setminus \{0,1\}$, so $M$ has at most $|\fun(\parf)|-2$ inequivalent $\parf$-representations. If $M$ is near-regular then it follows that this bound is exact, so assume $M$ is dyadic but not near-regular. Consider the forbidden minors for $\GF(4)$-representable matroids, determined by \citet{GGK}. The only three that are dyadic are $F_7^-$, $(F_7^-)^*$, and $P_8$. Therefore $M$ must have one of these as a minor. From the previous lemma it follows that $M$ is uniquely representable over $\dyadic$, and by combining this with Lemma \ref{lem:U24settlesGF3} we conclude that every representation of $M$ over a partial field $\parf$ is obtained by a homomorphism $\dyadic\rightarrow\parf$. Since $\phi(1) = 1$ we have $\phi(2) = \phi(1) + \phi(1) = 1+1$. Therefore this homomorphism is unique, which completes the proof.
\end{proof}

Note that the situation for $\sru$ matroids is more complex, as it depends on the number of roots of $x^2-x+1$ in the ring $\ring$ of the partial field $\parf = (R,G)$. If $R$ is a field this number will, of course, be $0$ or $2$, but for rings this is not necessarily true.

%%%%%%%%%%%%%%%%%%%%%%%%%%%%%%%%%%%%%%%%%%%%%%%%%%%%%%%%%%%%%%%%%%%%%
\subsection{Quinary matroids}\label{sec:GFfive}
%%%%%%%%%%%%%%%%%%%%%%%%%%%%%%%%%%%%%%%%%%%%%%%%%%%%%%%%%%%%%%%%%%%%%
In this subsection we combine the Lift Theorem, the Confinement Theorem, and the theory of universal partial fields to obtain a detailed description of the representability of 3-connected quinary matroids with a specified number of inequivalent representations over $\GF(5)$. First we deal with those quinary matroids that have no $U_{2,5}$- and no $U_{3,5}$-minor.
\begin{theorem}[\citet{SW96b}]\label{thm:noUtwofives}
  Let $M$ be a 3-connected matroid representable over some field. If $M$ has no $U_{2,5}$- and no $U_{3,5}$-minor, then $M$ is either binary or ternary.
\end{theorem}
It is probably not hard to prove this theorem using an argument similar to our proof of Theorem~\ref{thm:classification}. On combining this with Theorem \ref{thm:regnregdyadic} we obtain:

\begin{corollary}\label{cor:noUtwofive}
  Let $M$ be a 3-connected quinary matroid with no $U_{2,5}$- and no $U_{3,5}$-minor. Exactly one of the following holds:
  \begin{enumerate}
    \item $M$ is regular. In this case $M$ is uniquely representable over $\GF(5)$.
    \item $M$ is near-regular but not regular. In this case $M$ has exactly 3 inequivalent representations over $\GF(5)$.
    \item $M$ is dyadic but not near-regular. In this case $M$ is uniquely representable over $\GF(5)$.
  \end{enumerate}
\end{corollary}
% \begin{proof}
%   Only the second part does not follow directly from the previous theorem.
%   Let $\phi_2, \phi_3,\phi_4$ be homomorphisms $\nreg\rightarrow\GF(5)$ determined by $\phi_i(\alpha) = i$. This gives three inequivalent representations over $\GF(5)$. By Lemma~\ref{lem:U24settlesGF3} there are not more.
% \end{proof}
It follows that we only have to characterize those 3-connected quinary matroids that do have a $U_{2,5}$- or $U_{3,5}$-minor. The following lemma is another application of the Stabilizer Theorem.
\begin{lemma}[\citet{Whi96b}]\label{Utwofivestabilizer}$U_{2,5}$ and $U_{3,5}$ are $\GF(5)$-stabilizers for the class of 3-connected quinary matroids.
\end{lemma}
Now we introduce a hierarchy of partial fields, the \emph{Hydra-$k$ partial fields}\footnote{The Hydra is a many-headed mythological monster that grows back two heads whenever you cut off one. The most famous is the Lernaean Hydra, which was killed by Herakles.} $\hydra_1, \hydra_2, \ldots, \hydra_6$, such that the following theorem holds:
\begin{theorem}
  Let $M$ be a 3-connected, quinary matroid that has a $U_{2,5}$- or $U_{3,5}$-minor, and let $k \in \{1, \ldots, 6\}$. The following are equivalent:
  \begin{enumerate}
    \item $M$ is representable over $\hydra_k$;
    \item $M$ has at least $k$ inequivalent representations over $\GF(5)$.
  \end{enumerate}
\end{theorem}
First we sketch how to construct the Hydra-$k$ partial fields. For $k=1$ we obviously pick $\hydra_1 := \GF(5)$. For $k > 1$ we consider $\parf_k:=\bigotimes_{i=1}^k\GF(5)$. Let $\phi_i:\parf_k\rightarrow\GF(5)$ be the $i$th projection map, i.e. $\phi_i(x) = x_i$, and let ${\cal A}_k$ be the class of 3-connected $\parf_k$-matrices $A$ for which the $\phi_i(A)$, $i=1, \ldots, k$ are pairwise inequivalent. For $k \geq 3$ we need to invoke the Confinement Theorem; its use is summarized in the following lemma.

\begin{lemma}
  Let $k \geq 3$. Let $p \in \parf_k$, $p \not \in \{(0,\ldots,0),(1,\ldots,1)\}$ be such that three coordinates are equal. Then $p \not \in \Crat({\cal A}_k)$.
\end{lemma}

\begin{proof}
  Suppose there is an $A \in {\cal A}_k$ such that $p \in \Crat(A)$. Without loss of generality we assume that the first three coordinates of $p$ are equal. Let $R'$ be the subring of $GF(5)^k$ in which the first three coordinates are equal, and define $\parf_k' := (R',\parf_k^* \cap R')$. Then $\parf_k'$ is clearly an induced sub-partial field. Suppose
  \begin{align}
    \begin{bmatrix}
      1 & 1 & 1\\
      1 & p & q
    \end{bmatrix} \minorof A
  \end{align}
  is a $\parf_k$-matrix with $q \not \in \{p,(0,\ldots,0),(1,\ldots,1)\}$. Note that $q_1,q_2,q_3 \not \in \{0,1,p_1\}$. Hence two of $q_1,q_2,q_3$ must be equal. By permuting we may assume $q_1=q_2$. But then Lemma \ref{Utwofivestabilizer} implies that $\phi_1(A) \sim \phi_2(A)$, contradicting the definition of ${\cal A}_k$.

  It follows that
  \begin{align}
  \begin{bmatrix}
    1 & 1\\
    p & 1
  \end{bmatrix}
  \end{align}
  has no 3-connected 1-element extensions or coextensions. But then Theorem \ref{thm:confinement} implies that this matrix confines $A$ to $\parf_k'$, again contradicting the definition of ${\cal A}_k$.
\end{proof}

Now we let $\hydra_k := \lift_{{\cal A}_k}\parf_k$, as in Definition~\ref{def:liftclass}. The descriptions of $\hydra_k$ that we will give below were obtained from $\lift_{{\cal A}_k}\parf_k$ by computing a Gr\"obner basis over the integers for the ideal, and choosing a suitable set of generators.

Let $M$ be a 3-connected matroid having a $U_{2,5}$- or $U_{3,5}$-minor, and at least $k$ inequivalent representations over $\GF(5)$. Then $M = M[I\: A]$ for some $\parf_k$-matrix $A \in {\cal A}_k$. By Lemma~\ref{Utwofivestabilizer} every representation of a $U_{2,5}$- or $U_{3,5}$-minor of $M$ is in ${\cal A}_k$, from which it follows that $M$ is representable over $\hydra_k$.

For the converse we cannot rule out a priori that there exists an $\hydra_k$-representation $A'$ of $U_{2,5}$ such that $\{\phi_i(\phi(A')) \mid i = 1, \ldots, k\}$ contains fewer than $k$ inequivalent representations over $\GF(5)$. To prove that this degeneracy does not occur, one may simply check each normalized $\hydra_k$-representation of $U_{2,5}$. This is feasible because it turns out that all of $\hydra_1, \ldots, \hydra_6$ have a finite number of fundamental elements.

Note that the computations described in the preceding paragraphs are quite elaborate. Rather than reproducing those in the paper we have made the computer code available for download in a technical report \cite{PZ10report}.

With this background we proceed with the description of the partial fields and their properties. First Hydra-2. This turns out to be the Gaussian partial field, introduced in \cite[{Section 4.2}]{PZ08lift}. There we proved the following results:
\begin{lemma}[{\cite[{Lemma 4.12}]{PZ08lift}}]\label{lem:hydratwo}Let $M$ be a 3-connected matroid.
\begin{enumerate}
  \item\label{eq:onehydratwo} If $M$ has at least $2$ inequivalent representations over $\GF(5)$, then $M$ is representable over $\hydra_2$.
  \item\label{eq:twohydratwo} If $M$ has a $U_{2,5}$- or $U_{3,5}$-minor and $M$ is representable over $\hydra_2$, then $M$ has at least $2$ inequivalent representations over $\GF(5)$.
\end{enumerate}
\end{lemma}
\begin{theorem}[{\cite[{Theorem 4.14}]{PZ08lift}}]\label{thm:gauss}
Let $M$ be a 3-connected matroid with a $U_{2,5}$- or $U_{3,5}$-minor. The following are equivalent:
\begin{enumerate}
  \item \label{eq:gaussfivefive}$M$ has 2 inequivalent representations over $\GF(5)$;
  \item \label{eq:gaussTU} $M$ is $\gauss$-representable;
  \item \label{eq:gaussevery}$M$ has two inequivalent representations over $\GF(5)$, is representable over $\GF(p^2)$ for all primes $p \geq 3$, and over $\GF(p)$ when $p \equiv 1 \mod 4$.
\end{enumerate}
\end{theorem}
Next up is Hydra-3. We have
\begin{equation}
  \hydra_3 := (\Q(\alpha),\langle -1,\alpha, \alpha-1, \alpha^2-\alpha+1\rangle).
\end{equation}

\begin{lemma}\label{lem:H3fun}
\begin{align}\fun(\hydra_3) = \assoc\left\{1,\alpha,\alpha^2-\alpha+1,\frac{\alpha^2}{\alpha-1},\frac{-\alpha}{(\alpha-1)^2}\right\}.
\end{align}
\end{lemma}
\begin{proof}[Proof sketch]
  All fundamental elements are of the form $(-1)^s\alpha^x(\alpha-1)^y(\alpha^2-\alpha+1)^z$.
  The homomorphism $\phi:\hydra_3\rightarrow\hydra_2$ determined by $\phi(\alpha) = i$ yields $-2\leq y \leq 2$, since fundamental elements must map to fundamental elements and the norms of fundamental elements of $\hydra_2$ are between $1/2$ and $2$. Similarly, $\psi:\hydra_3\rightarrow\hydra_2$ determined by $\psi(\alpha) = 1-i$ yields $-2\leq x \leq 2$ and $\rho:\hydra_3\rightarrow\hydra_2$ determined by $\rho(\alpha) = \frac{1-i}{2}$ yields, together with the preceding bounds, $-3\leq z\leq 3$. This reduces the proof to a finite check. We refer to \cite{PZ10report} for the computations.
\end{proof}

\begin{lemma}\label{lem:hydrathree}Let $M$ be a 3-connected matroid.
\begin{enumerate}
  \item\label{eq:onehydrathree} If $M$ has at least $3$ inequivalent representations over $\GF(5)$, then $M$ is representable over $\hydra_3$.
  \item\label{eq:twohydrathree} If $M$ has a $U_{2,5}$- or $U_{3,5}$-minor and $M$ is representable over $\hydra_3$, then $M$ has at least $3$ inequivalent representations over $\GF(5)$.
\end{enumerate}
\end{lemma}
\begin{proof}[Proof sketch]
  Let $\psi:\hydra_3\rightarrow \bigotimes_{i=1}^3\GF(5)$ be determined by $\psi(\alpha) = (2,3,4)$. A finite check shows for all $\hydra_3$-matrices $A=\left[\begin{smallmatrix}1 & 1 & 1\\1 & p & q\end{smallmatrix}\right]$ representing $U_{2,5}$ that $|\{\phi_i(\psi(A))\mid i = 1, \ldots, 3\}| = 3$. For the computations we refer to \cite{PZ10report}. Together with Lemma~\ref{Utwofivestabilizer} this proves \eqref{eq:twohydrathree}$\Rightarrow$\eqref{eq:onehydrathree}.

  Let $\phi:\hydra_3\rightarrow\bigotimes_{i=1}^3 \GF(5)$ be determined by $\phi(\alpha) = (2,3,4)$. Then $\phi|_{\fun(\hydra_3)}:\fun(\hydra_3)\rightarrow\Crat({\cal A}_k)$ is a bijection and by Theorem~\ref{thm:liftclass} and Lemma~\ref{Utwofivestabilizer} it follows that all matroids in ${\cal A}_k$ are representable over $\hydra_3$. Again, the necessary computations can be found in \cite{PZ10report}. Together with Theorem~\ref{thm:regnregdyadic} this proves \eqref{eq:onehydrathree}$\Rightarrow$\eqref{eq:twohydrathree}.
\end{proof}

Next up is Hydra-4. From now on we omit the proof sketches since no new technicalities arise. All computations can be found in \cite{PZ10report}.
\begin{equation}
  \hydra_4 := (\Q(\alpha,\beta),\langle -1,\alpha,\beta,\alpha-1,\beta-1,\alpha\beta-1, \alpha+\beta-2\alpha\beta\rangle).
\end{equation}
There exists a homomorphism $\phi:\hydra_4\rightarrow\bigotimes_{k=1}^4 \GF(5)$ determined by $\phi(\alpha) = (2,3,3,4)$, $\phi(\beta) = (2,3,4,3)$.
\begin{lemma}\label{lem:H4fun}
\begin{align}
  \fun(\hydra_4) = \assoc\Big\{&1,\alpha, \beta, \alpha\beta,\tfrac{\alpha-1}{\alpha\beta-1},\tfrac{\beta-1}{\alpha\beta-1},-\tfrac{\alpha(\beta-1)}{\beta(\alpha-1)},
  \tfrac{(\alpha-1)(\beta-1)}{1-\alpha\beta},\notag\\
   &\tfrac{\alpha(\beta-1)^2}{\beta(\alpha\beta-1)},\tfrac{\beta(\alpha-1)^2}{\alpha(\alpha\beta-1)}\Big\}.
\end{align}
\end{lemma}
\begin{lemma}\label{lem:hydrafour}Let $M$ be a 3-connected matroid.
\begin{enumerate}
  \item\label{eq:onehydrafour} If $M$ has at least $4$ inequivalent representations over $\GF(5)$, then $M$ is representable over $\hydra_4$.
  \item\label{eq:twohydrafour} If $M$ has a $U_{2,5}$- or $U_{3,5}$-minor and $M$ is representable over $\hydra_4$, then $M$ has at least $4$ inequivalent representations over $\GF(5)$.
\end{enumerate}
\end{lemma}

Next up is Hydra-5.
\begin{align}
  \hydra_5 := (\Q(\alpha,\beta,\gamma),\langle& -1,\alpha,\beta,\gamma,\alpha-1,\beta-1,\gamma-1,\alpha-\gamma,\notag\\
  & \gamma-\alpha\beta,(1-\gamma)-(1-\alpha)\beta\rangle).
\end{align}
There exists a homomorphism $\phi:\hydra_5\rightarrow\bigotimes_{k=1}^5 \GF(5)$ determined by $\phi(\alpha) = (2,3,4,2,3)$, $\phi(\beta) = (3,2,3,4,2)$, $\phi(\gamma) = (3,2,3,4,4)$.
\begin{lemma}\label{lem:H5fun}
\begin{align}
  \fun(\hydra_5) = \assoc\Big\{&1,\alpha,\beta,\gamma,\tfrac{\alpha\beta}{\gamma},\tfrac{\alpha}{\gamma},\tfrac{(1-\alpha)\gamma}{\gamma-\alpha},\tfrac{(\alpha-1)\beta}{\gamma-1},\tfrac{\alpha-1}{\gamma-1},\tfrac{\gamma-\alpha}{\gamma-\alpha\beta},\notag\\
  & \tfrac{(\beta-1)(\gamma-1)}{\beta(\gamma-\alpha)},\tfrac{\beta(\gamma-\alpha)}{\gamma-\alpha\beta},\tfrac{(\alpha-1)(\beta-1)}{\gamma-\alpha},\notag\\
  & \tfrac{\beta(\gamma-\alpha)}{(1-\gamma)(\gamma-\alpha\beta)},\tfrac{(1-\alpha)(\gamma-\alpha\beta)}{\gamma-\alpha},\tfrac{1-\beta}{\gamma-\alpha\beta}
  \Big\}.
\end{align}
\end{lemma}

\begin{lemma}\label{lem:hydrafive}Let $M$ be a 3-connected matroid.
\begin{enumerate}
  \item\label{eq:onehydrafive} If $M$ has at least $5$ inequivalent representations over $\GF(5)$, then $M$ is representable over $\hydra_5$.
  \item\label{eq:twohydrafive} If $M$ has a $U_{2,5}$- or $U_{3,5}$-minor and $M$ is representable over $\hydra_5$, then $M$ has at least $5$ inequivalent representations over $\GF(5)$.
\end{enumerate}
\end{lemma}
Finally we consider $\hydra_6$. There exists a homomorphism $\phi:\hydra_5\rightarrow\bigotimes_{k=1}^6 \GF(5)$ determined by $\phi(\alpha) = (2,3,4,2,3,4)$, $\phi(\beta) = (3,2,3,4,2,4)$, $\phi(\gamma) = (3,2,3,4,4,2)$. It turns out that for every $\hydra_5$-representation $A'$ of $U_{2,5}$, $|\{\phi_i(\phi(A')) \mid i = 1, \ldots, 6\}| =6$. Therefore we define
\begin{align}
  \hydra_6 := \hydra_5
\end{align}
and immediately obtain the following strengthening of Lemma~\ref{lem:hydrafive}:
\begin{lemma}\label{lem:hydrasix}Let $M$ be a 3-connected matroid.
\begin{enumerate}
  \item\label{eq:onehydrasix} If $M$ has at least $5$ inequivalent representations over $\GF(5)$, then $M$ is representable over $\hydra_5$.
  \item\label{eq:twohydrasix} If $M$ has a $U_{2,5}$- or $U_{3,5}$-minor and $M$ is representable over $\hydra_5$, then $M$ has at least $6$ inequivalent representations over $\GF(5)$.
\end{enumerate}
\end{lemma}
We now have all ingredients for the proof of Theorem~\ref{thm:quinary} from the introduction.
\begin{proof}[Proof of Theorem~\ref{thm:quinary}]
  Let $M$ be a 3-connected quinary matroid. By Corollary~\ref{cor:noUtwofive} all of \eqref{quin:two}--\eqref{quin:five} hold when $M$ does not have a $U_{2,5}$- or $U_{3,5}$-minor. Therefore we may assume that $M$ does have a $U_{2,5}$- or $U_{3,5}$-minor.

  Statement~\eqref{quin:two} is \cite[{Theorem 4.12}]{PZ08lift}. For statement~\eqref{quin:three}, let $\field$ be a field, and let $p \in \field$ be an element that is not a root of the polynomials $x, x-1, x^2-x+1$. If $|\field|\geq 5$ then such an element must certainly exist. In that case $\phi:\hydra_3\rightarrow \field$ determined by $\phi(\alpha) = p$ is a nontrivial homomorphism.

  Statement~\eqref{quin:five} follows from Lemma~\ref{lem:hydrasix}.
\end{proof}
One could suspect that Theorem~\ref{thm:quinary}\eqref{quin:five} is true by observing that there is a bijection between the representations of $U_{2,5}$ in ${\cal A}_5$ and those in ${\cal A}_6$. But there seems to be no obvious reason why this bijection should extend to all $A \in {\cal A}_5$.

As a final remark we note that the partial fields $\hydra_k$ possess large automorphism groups, since permutations of coordinates in $\bigotimes_{i=1}^k \GF(5)$ must correspond with automorphisms of $\hydra_k$. Our representations of $\hydra_k$ obscure this fact, but expose other information in return. In \cite{PZ10report} we verify that the automorphism groups are isomorphic to $S_k$, the symmetric group on $k$ symbols, for $k \in \{1,2,3,4,6\}$.

%%%%%%%%%%%%%%%%%%%%%%%%%%%%%%%%%%%%%%%%%%%%%%%%%%%%%%%%%%%%%%%%%%%%%
\section{A number of questions and conjectures}\label{sec:questions}
%%%%%%%%%%%%%%%%%%%%%%%%%%%%%%%%%%%%%%%%%%%%%%%%%%%%%%%%%%%%%%%%%%%%%
The following conjecture links fundamental elements and universal partial fields.
\begin{conjecture}If $\parf_N$ has finitely many fundamental elements, then all $\parf_N$-representations of $N$ are equivalent.
\end{conjecture}
This conjecture cannot be strengthened by much. Consider the matroid $M[I\: A_3]$ from Table~\ref{tab:universal}, which is obtained from the Fano matroid by adding one element freely to a line. The homomorphism $\phi:\uniform_1^{(2)}\rightarrow\uniform_1^{(2)}$ determined by $x \mapsto x^2$ is not an automorphism. A related conjecture is the following:
\begin{conjecture}If $N$ is 3-connected then $N$ is a $\parf_N$-stabilizer for the class of $\parf_N$-representable matroids.
\end{conjecture}
Even if this is only true when $N$ is uniquely $\parf_N$-representable this conjecture would have important implications. For example a theorem by \citet{GGW06} would follow immediately and could, in fact, be strengthened.

Not all partial fields are universal. For instance, it is not hard to construct partial fields with homomorphisms to $\GF(3)$ different from the ones in Theorem~\ref{thm:classification}.
\begin{question}
  What distinguishes universal partial fields from partial fields in general?
\end{question}
We say that a partial field $\parf$ is \emph{level} if $\parf = \lift_{\cal A}\parf'$ for some partial field $\parf'$, where ${\cal A}$ is the class of $\parf'$-representable matroids.
\begin{question}\label{con:naturallevel}
  Under what conditions is $\parf_M$ level?
\end{question}
The converse of the latter question is also of interest.
\begin{question}\label{con:levelnatural}
  When is a level partial field also universal?
\end{question}
As shown in Table~\ref{tab:universal}, several known level partial fields are universal. The notable omissions in that table are the Hydra-$k$ partial fields for $k \geq 3$. We do not know if these are universal. The problem here is that many partial fields have exactly $k$ homomorphisms to $\GF(5)$, and all examples that we tried from Mayhew and Royle's catalog of small matroids~\cite{MR08} turned out to have slightly different universal partial fields.

A somewhat weaker statement is the following. Let ${\cal M}$ be a class of matroids. A partial field $\parf$ is \emph{${\cal M}$-universal} if, for every partial field $\parf'$ such that every matroid in ${\cal M}$ is $\parf'$-representable, there exists a homomorphism $\phi:\parf\rightarrow\parf'$.
\begin{conjecture}\label{con:leveluniversal}
  Let ${\cal M}$ be the set of all $\parf$-representable matroids, where $\parf$ is a level partial field. Then $\parf$ is ${\cal M}$-universal.
\end{conjecture}
As mentioned before, the Settlement Theorem is reminiscent of the theory of free expansions from \citet{GOVW02}. We offer the following conjecture:
\begin{conjecture}
  Let $M$ be a representable matroid. Then $M\delete e$ settles $M$ if and only if $e$ is fixed in $M$.
\end{conjecture}
Define the set
\begin{align}
  \chi_{\parf} := \{ \parf_M \mid M \textrm{ 3-connected, }\parf\textrm{-representable matroid}\}.
\end{align}
Whittle's classification, Theorem~\ref{thm:classification}, amounts to
\begin{align}
  \chi_{\GF(3)} = \{\uniform_0, \uniform_1, \dyadic, \psru, \GF(3)\}.
\end{align}
It is known that $\chi_{\GF(4)}$ is infinite, but it might be possible to determine $\chi_{\parf}$ for other partial fields. A first candidate might be $\GF(4)\otimes\GF(5)$, which is the class of \emph{golden ratio} matroids. Unfortunately our proof of Theorem~\ref{thm:classification} can not be adapted to this case, since we no longer have control over the set of fundamental elements. We outline a different approach. For all $\parf_M \in \chi_{\parf}$, there exists a ``totally free'' matroid $N\minorof M$ that settles $M$. Moreover, it is known that all totally free $\parf$-representable matroids can be found by an inductive search. Clearly $\ring_M \cong \ring_N/I_{N,M}$ for some ideal $I_{N,M}$. The main problem, now, consists of finding the possible ideals $I_{N,M}$.
\begin{conjecture}
  If $N = M\delete e$, $N, M$ are 3-connected, and $N$ settles $M$, then
  $I_{N,M}$ is an ideal generated by relations $p - q$, where $p,q \in \crat(N)$.
\end{conjecture}
The conjecture holds for all 3-connected 1-element extensions of a 6-element, rank-3 matroid. One example is $N = U_{3,6}$ and $M = \Phi_3^+$, the rank-3 free spike with tip.

\paragraph*{Acknowledgements}We thank Gordon Royle for his help in querying the catalog of small matroids (\citet{MR08}). One of these queries resulted in the matroid $M[I\: A_1]$ from Table~\ref{tab:universal}. We thank the two anonymous referees for thoroughly reading our manuscript and providing many helpful pointers for improvement.

\appendix

%%%%%%%%%%%%%%%%%%%%%%%%%%%%%%%%%%%%%%%%%%%%%%%%%
\section{A catalog of partial fields}\label{app:catalog}
%%%%%%%%%%%%%%%%%%%%%%%%%%%%%%%%%%%%%%%%%%%%%%%%%

\begin{figure}[p]
  \center
  \includegraphics{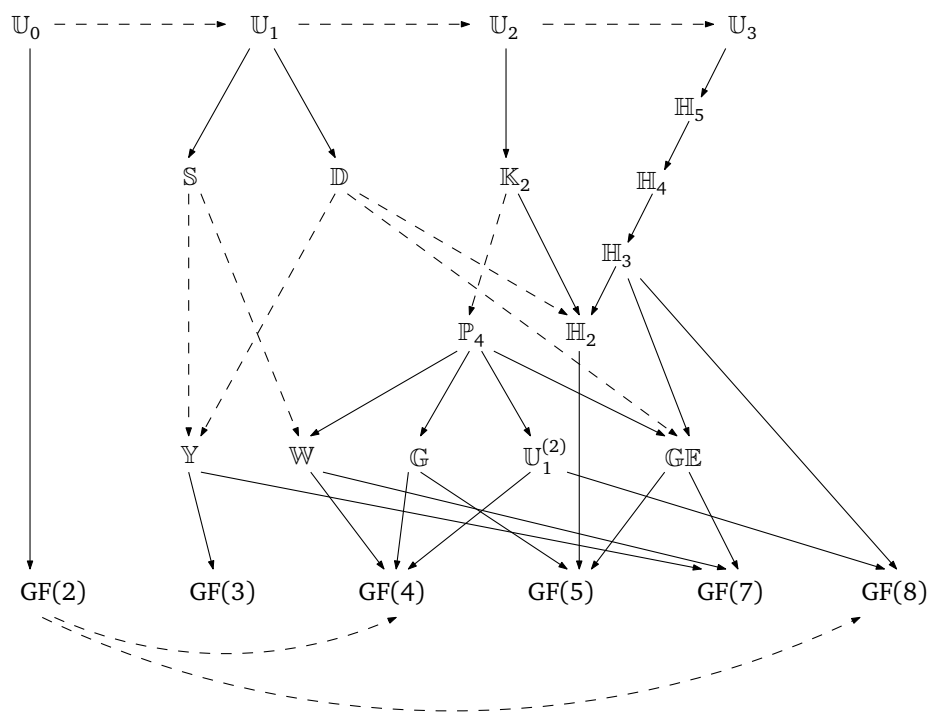}
  \caption{Some partial fields and their homomorphisms. A (dashed) arrow from $\parf'$ to $\parf$ indicates that there is an (injective) homomorphism $\parf'\rightarrow \parf$.}\label{fig:pfcat2}
\end{figure}

In this appendix we summarize all partial fields that have appeared in this paper and in \cite{PZ08lift}, and some of their basic properties. Like rings, partial fields form a category. The regular partial field, $\reg$, has a homomorphism to every other partial field. The references for the lift partial fields ($\lift\parf$) point to papers that first observed that the two partial fields carry the same set of matroids. For the actual computation of $\lift\parf$ we refer to \cite[Section 5]{PZ08lift}.

\begin{description}
  \item[The regular partial field,] $\reg$:
  \begin{itemize}
    \item $\reg = (\Z, \{-1,0,1\})$;
    \item $\fun(\reg) = \{0,1\}$;
    \item There is a homomorphism to every partial field $\parf$ \cite[Theorem 2.29]{PZ08lift};
    \item Isomorphic to $\lift(\GF(2)\times\GF(3))$ \cite{Tut65};
%    \item Isomorphic to $\dowling\GF(2)$ (Section \ref{ssec:dowlingex});
    \item There are finitely many excluded minors for $\reg$-representability \cite{Tut65}.
  \end{itemize}
  \item[The near-regular partial field,] $\nreg$:
  \begin{itemize}
    \item $\nreg = \left(\Z[\alpha, \tfrac{1}{1-\alpha}, \tfrac{1}{\alpha}], \langle -1, \alpha, 1-\alpha \rangle\right)$, where $\alpha$ is an indeterminate;
    \item $\fun(\nreg) = \assoc\{1,\alpha\} = \left\{0,1,\alpha,1-\alpha,\frac{1}{1-\alpha},\frac{\alpha}{\alpha-1},\frac{\alpha-1}{\alpha},\frac{1}{\alpha}\right\}$\\ \cite[Lemma 4.4]{PZ08lift};
    \item There is a homomorphism to every field with at least three elements \cite[Theorem 1.4]{Whi97};
    \item Isomorphic to $\lift(\GF(3)\times\GF(8))$ and $\lift(\GF(3)\times\GF(4)\times\GF(5))$ \cite[Theorem 1.4]{Whi97};
%    \item Isomorphic to $\dowling\psru$ (Section \ref{ssec:dowlingex});
    \item There are finitely many excluded minors for $\nreg$-representability \cite{HMZ09}.
  \end{itemize}
  \item[The $k$-uniform partial field,] $\uniform_k$:
  \begin{itemize}
    \item $\uniform_k = (\Q(\alpha_1, \ldots, \alpha_k),\langle U_k\rangle)$, where
     \begin{align*}
       U_k := \{\, x - y \mid x,y \in \{0,1,\alpha_1,\ldots,\alpha_k\}, x \neq y\,\},
     \end{align*} and $\alpha_1,\ldots,\alpha_k$ are  indeterminates;
    \item Introduced by \citet{Sem97} as the \emph{$k$-regular} partial field;
    \item \citet{Sem97} proved that
    \begin{align*}
      \fun(\uniform_k) = & \left\{ \tfrac{a-b}{c-b} \,\Big|\, a,b,c \in \{0,1,\alpha_1,\ldots,\alpha_k\}, \textrm{ distinct}\right\} \cup\notag\\
      & \left\{ \tfrac{(a-b)(c-d)}{(c-b)(a-d)} \,\Big|\, a,b,c,d \in \{0,1,\alpha_1,\ldots,\alpha_k\}, \textrm{ distinct}\right\};
    \end{align*}
    \item There is a homomorphism to every field with at least $k+2$ elements \cite[Proposition 3.1]{Sem97};
    \item Finitely many excluded minors for $\uniform_k$-re\-pre\-sen\-ta\-bi\-li\-ty are $\uniform_{k'}$-re\-pre\-sen\-table for some $k' > k$ \cite{OSV00}.
  \end{itemize}
  \item[The sixth-roots-of-unity ($\sru$) partial field,] $\psru$:
  \begin{itemize}
    \item $\psru = (\Z[\zeta], \langle \zeta \rangle)$, where $\zeta$ is a root of $x^2-x+1=0$.
    \item $\fun(\psru) = \assoc\{1,\zeta\} = \{0,1,\zeta,1-\zeta\}$;
    \item There is a homomorphism to $\GF(3)$, to $\GF(p^2)$ for all primes $p$, and to $\GF(p)$ when $p \equiv 1 \mod 3$ \cite[Theorem 2.30]{PZ08lift};
    \item Isomorphic to $\lift(\GF(3)\times\GF(4))$ \cite[Theorem 1.2]{Whi97};
    \item There are finitely many excluded minors for $\psru$-representability \cite[Corollary 1.4]{GGK}.
  \end{itemize}
  \item[The dyadic partial field,] $\dyadic$:
  \begin{itemize}
    \item $\dyadic = \left(\Z[\tfrac{1}{2}], \langle -1,2\rangle\right)$;
    \item $\fun(\dyadic) = \assoc\{1,2\} = \{0,1,-1,2,1/2\}$ \cite[Lemma 4.2]{PZ08lift};
    \item There is a homomorphism to every field that does not have characteristic two \cite[Theorem 1.1]{Whi97};
    \item Isomorphic to $\lift(\GF(3)\times\GF(5))$ \cite[Theorem 1.1]{Whi97};
%    \item Isomorphic to $\dowling\GF(3)$ (Section \ref{ssec:dowlingex});
  \end{itemize}
  \item[The union of $\sru$ and dyadic,] $\splittable$:
  \begin{itemize}
    \item $\splittable = (\Z[\zeta,\tfrac{1}{2}], \langle -1, 2, \zeta\rangle)$, where $\zeta$ is a root of $x^2-x+1=0$;
    \item $\fun(\splittable) = \assoc\{1,2,\zeta\} = \{0,1,-1,2,1/2,\zeta,1-\zeta\}$ \cite[Lemma 4.6]{PZ08lift};
    \item There is a homomorphism to $\GF(3)$, to $\GF(p^2)$ for all odd primes $p$, and to $\GF(p)$ when $p \equiv 1 \mod 3$ \cite[Theorem 4.7]{PZ08lift};
    \item Isomorphic to $\lift(\GF(3)\times\GF(7))$ \cite[Theorem 1.3]{Whi97}.
  \end{itemize}
  \item[The $2$-cyclotomic partial field,] $\cyclo_2$:
  \begin{itemize}
    \item $\cyclo_2 = \left(\Q(\alpha), \langle -1, \alpha, \alpha-1,\alpha+1\rangle\right)$, where $\alpha$ is an indeterminate;
    \item $\fun(\cyclo_2) = \assoc\{1,\alpha,-\alpha,\alpha^2\}$ \cite[Lemma 4.16]{PZ08lift};
    \item There is a homomorphism to $\GF(q)$ for $q \geq 4$ \cite[Lemma 4.14]{PZ08lift};
    \item Isomorphic to $\lift(\GF(4)\times\gauss)$ \cite[Theorem 4.17]{PZ08lift};
%    \item Isomorphic to $\dowling\golrat$ (Section \ref{ssec:dowlingex}).
  \end{itemize}
  \item[The $k$-cyclotomic partial field,] $\cyclo_k$:
  \begin{itemize}
      \item $\cyclo_k = \left(\Q(\alpha), \langle -1, \alpha, \alpha-1, \alpha^2-1, \ldots, \alpha^k-1\rangle\right)$, where $\alpha$ is an indeterminate;
      \item $\cyclo_k = (\Q(\alpha), \langle\{-1\}\cup \{\Phi_j(\alpha)\mid j = 0, \ldots, k\}\rangle)$, where $\Phi_0(\alpha) = \alpha$ and $\Phi_j$ is the $j$th \emph{cyclotomic polynomial} \cite[Lemma 4.15]{PZ08lift};
      \item There is a homomorphism to $\GF(q)$ for $q \geq k+2$ \cite[Lemma 4.14]{PZ08lift}.
  \end{itemize}
  \item[The ``Dowling lift'' of $\GF(4)$,] $\dowfour$:
  \begin{itemize}
    \item $\dowfour := \left(\Z[\zeta, \tfrac{1}{1+\zeta}], \langle -1, \zeta, 1+\zeta\rangle\right)$, where $\zeta$ is a root of $x^2-x+1=0$;
    \item $\fun(\dowfour) = \assoc\{1,\zeta, \zeta^2\} = \big\{0,1, \zeta, \bar\zeta, \zeta^2, \bar\zeta^{2}, \zeta+1, (\zeta+1)^{-1}, (\bar\zeta+1)^{-1}, \bar\zeta+1
                     \big\}$ \cite[Lemma 2.5.37]{vZ09};
%    \item Isomorphic to $\dowling\GF(4)$ (Section \ref{ssec:dowlingex});
    \item There is a homomorphism to every field with an element of multiplicative order $3$ \cite[Theorem 3.2.8]{vZ09}.
  \end{itemize}
  \item[The Gersonides partial field,] $\ger$:
  \begin{itemize}
    \item $\ger = \left( \Z[\tfrac{1}{2}, \tfrac{1}{3}], \langle -1,2,3\rangle\right)$;
    \item $\fun(\ger) = \assoc\{1,2,3,4,9\}$ \cite[Lemma 2.5.40]{vZ09};
    \item There is a homomorphism to every field that does not have characteristic two or three \cite[2.5.39]{vZ09}.
  \end{itemize}
  \item[The partial field ] $\parf_4$:
  \begin{itemize}
    \item $\parf_4 = (\Q(\alpha), \langle -1,\alpha,\alpha-1,\alpha+1,\alpha-2\rangle)$, where $\alpha$ is an indeterminate;
    \item $\fun(\parf_4) = \assoc\{1,\alpha,-\alpha,\alpha^2,\alpha-1,(\alpha-1)^2\}$ \cite[Lemma 2.5.43]{vZ09};
    \item There is a homomorphism to every field with at least four elements \cite[Lemma 2.5.43]{vZ09}.
  \end{itemize}
  \item[The Gaussian partial field,] $\gauss$:
  \begin{itemize}
    \item $\gauss = \left(\Z[i,\tfrac{1}{2}], \langle i, 1-i\rangle\right)$, where $i$ is a root of $x^2+1=0$;
    \item $\fun(\gauss) = \assoc\{1,2,i\} = \left\{0,1,-1,2,\tfrac{1}{2},i,i+1,\tfrac{i+1}{2},1-i,\tfrac{1-i}{2},-i\right\}$\\ \cite[Lemma 4.10]{PZ08lift};
    \item There is a homomorphism to $\GF(p^2)$ for all primes $p \geq 3$, and to $\GF(p)$ when $p \equiv 1 \mod 4$ \cite[Theorem 4.13]{PZ08lift};
    \item A matroid is $\gauss$-representable if and only if it is dyadic or has at least two inequivalent $\GF(5)$-representations \cite[Lemma 4.12]{PZ08lift};
%    \item Isomorphic to $\lift_\mathcal{A}(\GF(5)\times\GF(5))$, where $\mathcal{A}$ is the set of $\GF(5)\times\GF(5)$-matrices whose two projections are inequivalent (Section \ref{sec:confapp});
%    \item Isomorphic to $\dowling\GF(5)$ (Section \ref{ssec:dowlingex}).
  \end{itemize}
  \item[The Hydra-$3$ partial field,] $\hydra_3$:
  \begin{itemize}
    \item $\hydra_3 = (\Q(\alpha),\langle -1, \alpha, 1-\alpha, \alpha^2-\alpha+1\rangle)$, where $\alpha$ is an indeterminate;
    \item $\fun(\hydra_3) = \assoc\left\{1,\alpha,\alpha^2-\alpha+1,\tfrac{\alpha^2}{\alpha-1},\tfrac{-\alpha}{(\alpha-1)^2}\right\}$ (Lemma \ref{lem:H3fun});
    \item There is a homomorphism to every field with at least five elements (Theorem \ref{thm:quinary});
    \item A matroid is $\hydra_3$-representable if and only if it is regular or has at least three inequivalent $\GF(5)$-representations (Lemma \ref{lem:hydrathree}).
  \end{itemize}
  \item[The Hydra-$4$ partial field,] $\hydra_4$:
  \begin{itemize}
    \item $\hydra_4 = (\Q(\alpha,\beta),\langle -1,\alpha,\beta,\alpha-1,\beta-1,\alpha\beta-1, \alpha+\beta-2\alpha\beta\rangle)$, where $\alpha$, $\beta$ are indeterminates;
    \item $\fun(\hydra_4) = \assoc\Big\{1,\alpha, \beta, \alpha\beta,\tfrac{\alpha-1}{\alpha\beta-1},\tfrac{\beta-1}{\alpha\beta-1},-\tfrac{\alpha(\beta-1)}{\beta(\alpha-1)},
  \tfrac{(\alpha-1)(\beta-1)}{1-\alpha\beta},\\
   \tfrac{\alpha(\beta-1)^2}{\beta(\alpha\beta-1)}$, $\tfrac{\beta(\alpha-1)^2}{\alpha(\alpha\beta-1)}\Big\}$ (Lemma \ref{lem:H4fun});
    \item There is a homomorphism to every field with at least five elements;
    \item A matroid is $\hydra_4$-representable if and only if it is near-regular or has at least four inequivalent $\GF(5)$-representations (Lemma \ref{lem:hydrafour}).
  \end{itemize}
  \item[The Hydra-$5$ partial field,] $\hydra_5 = \hydra_6$:
  \begin{itemize}
    \item $\hydra_5 = (\Q(\alpha,\beta,\gamma),\langle -1,\alpha,\beta,\gamma,\alpha-1,\beta-1,\gamma-1,\alpha-\gamma,
     \gamma-\alpha\beta$,\\$(1-\gamma)-(1-\alpha)\beta\rangle)$, where $\alpha$, $\beta$, $\gamma$ are indeterminates;
    \item $\fun(\hydra_5) = \assoc\Big\{1,\alpha,\beta,\gamma,\tfrac{\alpha\beta}{\gamma},\tfrac{\alpha}{\gamma},\tfrac{(1-\alpha)\gamma}{\gamma-\alpha},\tfrac{(\alpha-1)\beta}{\gamma-1},\tfrac{\alpha-1}{\gamma-1},\tfrac{(\beta-1)(\gamma-1)}{\beta(\gamma-\alpha)},$\\
  $\tfrac{\gamma-\alpha}{\gamma-\alpha\beta},\tfrac{\beta(\gamma-\alpha)}{\gamma-\alpha\beta},\tfrac{(\alpha-1)(\beta-1)}{\gamma-\alpha},
  \tfrac{\beta(\gamma-\alpha)}{(1-\gamma)(\gamma-\alpha\beta)},\tfrac{(1-\alpha)(\gamma-\alpha\beta)}{\gamma-\alpha},\tfrac{1-\beta}{\gamma-\alpha\beta}
  \Big\}$ (Lemma \ref{lem:H5fun});
    \item There is a homomorphism to every field with at least five elements;
    \item A matroid is $\hydra_5$-representable if and only if it is near-regular or has at least \emph{six} inequivalent $\GF(5)$-representations (Lemma \ref{lem:hydrafive}).
  \end{itemize}
  \item[The near-regular partial field modulo two,] $\nreg^{(2)}$:
  \begin{itemize}
    \item $\nreg^{(2)} = (\GF(2)(\alpha), \langle\alpha, 1+\alpha\rangle)$, where $\alpha$ is an indeterminate;
    \item $\fun(\nreg^{(2)}) = \{0,1\}\cup \assoc \left\{\alpha^{2^k} \mid k \in \N\right\}$ \cite[Lemma 2.5.46]{vZ09};
    \item There is a homomorphism to $\GF(2^k)$ for all $k \geq 2$ \cite[Lemma 2.5.45]{vZ09}.
  \end{itemize}
  \item[The golden ratio partial field,] $\golrat$:
  \begin{itemize}
    \item $\golrat = (\Z[\tau], \langle -1, \tau \rangle)$, where $\tau$ is the positive root of $x^2-x-1=0$;
    \item $\fun(\golrat) = \assoc\{1,\tau\} = \{0,1,\tau, -\tau, 1/\tau, -1/\tau, \tau^2, 1/\tau^2\}$\\ \cite[Lemma 4.8]{PZ08lift}
    \item There is a homomorphism to $\GF(5)$, to $\GF(p^2)$ for all primes $p$, and to $\GF(p)$ when $p \equiv \pm 1 \mod 5$ \cite[Theorem 4.9]{PZ08lift};
    \item Isomorphic to $\lift(\GF(4)\times\GF(5))$ \cite[Theorem 4.9]{PZ08lift}.
  \end{itemize}
\end{description}

  \renewcommand{\Dutchvon}[2]{#1}
  \bibliography{matbib}
  \bibliographystyle{elsarticle-num-names}
  % Let WinEdt know where the bib file is, to make \cite{} yield a popup menu
  %GATHER{matbib.bib}
\end{document}